\documentclass[12pt,reqno]{article}
\usepackage[utf8]{inputenc}
\usepackage{amsmath}
\usepackage{amsfonts}
\usepackage{amssymb}
\usepackage{amsthm}	
\usepackage{cases}
\usepackage{subeqnarray}
\usepackage{cite}
\usepackage[colorlinks]{hyperref}
\usepackage{graphics}
\usepackage{graphicx}
\usepackage{fancyref}
\usepackage{epstopdf}
\usepackage{float}
\usepackage{enumerate} 
\usepackage{bm}
\usepackage{titlesec}
\usepackage[titletoc,toc,title]{appendix}
\usepackage{mathrsfs}

\makeatletter
\allowdisplaybreaks \numberwithin{equation}{section}
\newtheorem{theorem}{Theorem}[section]
\newtheorem{corollary}{Corollary}[section]
\newtheorem{lemma}{Lemma}[section]
\newtheorem{proposition}{Proposition}[section]

\newtheorem{definition}{Definition}[section]

\newcommand{\D}{\mathrm{d}}
\newcommand{\Real}{\mathbb{R}}

\newcommand{\Int}{\mathrm{Int}}

\newcommand{\oOmega}{\overline{\Omega}}

\newcommand{\cB}{\mathcal{B}}

\newcommand{\cD}{\mathcal{D}}

\newcommand{\cF}{\mathcal{F}}
\newcommand{\cG}{\mathcal{G}}
\newcommand{\cL}{\mathcal{L}}

\newcommand{\cQ}{\mathcal{Q}}
\newcommand{\cR}{\mathcal{R}}

\newcommand{\cV}{\mathcal{V}}

\newcommand{\cX}{\mathcal{X}}

\newcommand{\cM}{\mathcal{M}}
\newcommand{\cN}{\mathcal{N}}
\newcommand{\bF}{\mathbb{F}}
\newcommand{\bO}{\mathbb{O}}
\newcommand{\bP}{\mathbb{P}}
\newcommand{\bQ}{\mathbb{Q}}

\newcommand{\bU}{\mathbb{U}}

\newcommand{\bX}{\mathbb{X}}
\newcommand{\bY}{\mathbb{Y}}
\newcommand{\bZ}{\mathbb{Z}}

\newcommand{\tn}{\tilde{n}}

\newcommand{\R}{\mathcal{R}_0}

\newcommand{\bme}{\bm{e}}
\newcommand{\bmi}{\bm{i}}

\newcommand{\bmr}{\bm{r}}
\newcommand{\bmu}{\bm{u}}
\newcommand{\bmv}{\bm{v}}
\newcommand{\bmw}{\bm{w}}

\newcommand{\bmkp}{\bm{\kappa}}
\newcommand{\bmnu}{{\bm{\nu}}}

\newcommand{\Realparts}{{\rm Re}}

\begin{document}

\title{Asymptotic behavior of the basic reproduction ratio for periodic reaction-diffusion systems
\date{\empty}
\author{ Lei Zhang $^{a,b}$, Xiao-Qiang Zhao $^{b}$	\\
	{\small a Department of Mathematics, Harbin Institute of Technology at Weihai,}\\
	{\small Weihai, Shandong 264209, China.}\\
	 {\small b Department of Mathematics and Statistics, Memorial University of Newfoundland,}\\
	 {\small St. John's, NL A1C 5S7, Canada.}\\
}
}\maketitle
\begin{abstract}
	This  paper is devoted to the study of asymptotic behavior of the basic reproduction ratio for periodic reaction-diffusion systems in the case of small and large diffusion coefficients. We first establish the continuity of the basic reproduction ratio with respect to parameters by developing the theory of resolvent positive operators. Then we investigate the limiting profile of the principal eigenvalue of an associated periodic eigenvalue	problem for large diffusion coefficients. We then obtain the asymptotic behavior of the basic reproduction ratio as the diffusion coefficients go to zero and infinity, respectively. We also investigate the limiting behavior of positive periodic solution for periodic and cooperative reaction-diffusion systems with the Neumann boundary condition when the diffusion coefficients are large enough. Finally, we apply these results to a reaction-diffusion model of Zika virus transmission.
	
\end{abstract}
\section{Introduction}
The basic reproduction ratio $\R$ is one of the most valuable threshold quantities in population 
dynamics (see, e.g., \cite{Cushing2016, van2017, zhao2017dynamical}
and references therein).  In epidemiology, $\R$ is the expected number of secondary cases produced, in a completely 
susceptible population, by a typical infective individual.
Mathematically, the sign of $\R - 1$ can be  determined by  the stability of the zero solution for a linear system derived from the linearization at the zero solution or a disease-free equilibrium. This observation has been confirmed for various autonomous, periodic, and almost periodic
evolution equation models with or without time-delay (see, e.g., \cite{van2002reproduction,bacaer2006epidemic,wang2008threshold,thieme2009spectral,wang2012basic,wang2013basic,zhao2017basic,liang2019basic,liang2019principal,qiang2020basic}). Since there is no explicit formula of $\R$
for general periodic systems,  it is important to explore its  properties
qualitatively. 

Reaction–diffusion systems are widely used to study the spatial dynamics in population biology (see, e.g., \cite{cantrell2004spatial,fife1979mathematical,zhao2017dynamical}).
Let $\Omega \subset \Real^N$ be a bounded domain with smooth boundary $\partial \Omega$ and $\R(D)$ be the basic reproduction ratio associated with  the following scalar reaction-diffusion equation:
\begin{equation}\label{equ:R0:scalar:sys}
\begin{cases}
\frac{\partial u}{\partial t}=D \Delta u  -\gamma(x) u + \beta(x) u, & x \in \Omega,~t>0,\\
\frac{\partial u}{\partial \bm{\nu}}=0, &x \in \partial \Omega, ~t>0.
\end{cases}
\end{equation}
Here $D$ is the diffusion coefficient,  $\beta(x)$ and $\gamma(x)$ are the disease transmission and recovery rates, respectively. 
Allen, et al.  \cite{allen2008asymptotic} showed that
$$
\lim\limits_{D \rightarrow 0} \R(D) = \max_{ x \in \oOmega} \frac{\beta (x)}{\gamma (x)}  \quad 
\text{ and } \quad 
\lim\limits_{D \rightarrow +\infty} \R(D) = \frac{\int_{\Omega} \beta (x) \D x}{\int_{\Omega} \gamma(x) \D x}.
$$
Recently, Magal, Webb and Wu \cite{magal2019basic}, and Chen and Shi \cite{chen2020asymptotic} generalized such results to autonomous reaction-diffusion systems.
There are also some related works for  the patch models and scalar nonlocal dispersal equation models (see, e.g., \cite{allen2007Asymptotic,yang2019dynamics,gao2019travel,gao2020fast}).

In view of  population models with seasonality, 
 a natural question is whether these results on the asymptotic behavior of $\R$ can be extended to time-periodic reaction-diffusion systems.
Peng and Zhao\cite{peng2012reaction,peng2015effects} gave a  confirmative answer for scalar equations.  The purpose of this paper is to study the 
limiting profile of $\R$ associated with  general periodic systems 
for small and large  diffusion coefficients. It turns out that as the diffusion coefficients go to zero, $\R$ tends to the maximum of $\R(x)$ associated with the periodic reaction ODE systems with parameter $x \in \oOmega$ in the case of Dirichlet, Neumann and Robin boundary conditions, and that as the diffusion coefficients go to infinity, $\R$ tends to zero in the case of the Dirichlet and Robin boundary conditions, while $\R$ tends to the basic reproduction ratio of the periodic ODE system derived from the spatial average of the reaction terms in the case of the Neumann boundary condition.

To archive our purpose, we  first develop the theory of  resolvent positive operators under the setting of Thieme \cite{thieme2009spectral}, 
and then prove  the continuity of $\R$  with respect to parameters.  This enables us to reduce the problem on the limiting profile of 
$\R$ into that of the principal eigenvalue associated with linear periodic cooperative systems.  For a scalar elliptic eigenvalue problem, the limiting profile of the principal eigenvalue can be derived easily from  the standard variational formula (see, e.g., \cite{cantrell2004spatial}).
Dancer \cite{dancer2009principal} and Lam and Lou \cite{lam2016asymptotic} generalized such a result to cooperative elliptic systems with small diffusion coefficients. For a scalar time-periodic parabolic equation, Hutson, Mischaikow and Pol{\'a}{\v{c}}ik \cite{hutson2001evolution} and Peng and Zhao \cite{peng2015effects} studied the asymptotic behavior of the principal eigenvalue as the diffusion and advection coefficients go to zero and infinity, respectively.  More recently, Bai and He \cite{bai2020asymptotic} extended the results in \cite{lam2016asymptotic} to  periodic parabolic systems for small diffusion coefficients.  
It remains to  explore the asymptotic behavior of the principal eigenvalue
of these periodic systems for large diffusion coefficients. 
When the Poincar\'e (period)  map,  which is a square matrix, of the 
spatially averaged ODE system  is  irreducible, we solve this problem 
for large diffusion coefficients by adapting  the arguments in \cite{hutson2001evolution,peng2015effects}.  When such a matrix 
is reducible, we prove the same result by establishing 
the relationship between the block of the Poincar\'e map (matrix) and 
that  of the coefficient matrix of the corresponding periodic ODE system. 

It is also interesting to interpret these limiting results in terms of  the principal eigenvalue. Let us define the principal eigenvalue for linear periodic ODE systems in the same way as that for periodic reaction-diffusion 
systems. It turns out that as the diffusion coefficients go to zero, the principal eigenvalue of periodic reaction-diffusion systems tends to the minimum of the principal eigenvalues of the periodic reaction ODE systems with  parameter $x$  in  the case of Dirichlet, Neumann and Robin boundary conditions, and that as the diffusion coefficients go to infinity, it tends to infinity in the case of  the Dirichlet and Robin boundary conditions, while it tends to the principal eigenvalue of the periodic ODE system derived from the spatial average of  reaction terms in the case of the Neumann boundary condition.

Note that the disease-free periodic solution varies with  the diffusion coefficients in many periodic epidemic models. Thus,  it is natural to investigate the limiting behavior of the positive periodic solution when the diffusion coefficients are small and large enough, respectively. 
For reaction-diffusion systems subject to the Neumann boundary condition, Conway, Hoff and Smoller \cite{conway1978large}, Hale \cite{hale1986large}, Hale and Rocha \cite{hale1987varying},
 and Cantrell, Cosner and Hutson \cite{cantrell1996ecological}  showed that the solutions are asymptotic to those of an ODE when the diffusion coefficients are large. Hale and Sakamoto\cite{hale1989shadow} and Hutson, Mischaikow and Pol{\'a}{\v{c}}ik \cite{hutson2001evolution} also found  that the dynamics of reaction-diffusion systems with the Neumann boundary condition approximates that of the associated shadow systems as the diffusion coefficients  partially tend to infinity. Lam and Lou \cite{lam2016asymptotic} proved that the positive steady state of a reaction-diffusion system converges uniformly to the equilibrium of the corresponding kinetic system as the diffusion coefficients go to zero. Recently, Bai and He \cite{bai2020asymptotic} extended such results in \cite{lam2016asymptotic} to periodic systems for small diffusion coefficients. 

In order to apply  our developed theory of the asymptotic behavior of $\R$ to periodic reaction-diffusion models,  we further study the limiting behavior of the positive periodic solution of periodic and cooperative reaction-diffusion systems subject to the Neumann boundary condition for large diffusion coefficients. For this purpose, we proceed with two steps. The first step is to show that the spatial average of the positive periodic solution converges to the positive periodic solution of the spatially averaged ODE system as the diffusion coefficients go to infinity.  The second step is to prove that the positive periodic solution approximates its spatial average when the diffusion coefficients are large enough.
It turns out that as the diffusion coefficients go to zero, the positive periodic solution approaches the positive periodic solution of the periodic reaction ODE systems with  parameter $x$, and that as the diffusion coefficients go to infinity, it tends to the positive periodic solution of the periodic ODE system derived from the spatial  average of reaction terms.

 The remaining part of this paper is organized as follows. 
In the next section, we present some basic properties of resolvent positive operators, and prove the continuity of the basic reproduction ratio with respect to parameters for abstract  periodic systems. In section \ref{sec:AB}, we study the asymptotic behavior of the principal eigenvalue for periodic cooperative reaction-diffusion systems with large diffusion coefficients. In section \ref{sec:R0}, we establish the results on  the limiting 
profile  of the basic reproduction ratio as the diffusion coefficients go to zero and infinity, respectively. In section \ref{sec:nonlinear}, we investigate the limiting behavior of the positive periodic solution of periodic and cooperative reaction-diffusion systems with the Neumann boundary condition when the diffusion coefficients are large enough. In section \ref{example},  
as an  illustrative example, we 
apply these  analytic  results to a reaction-diffusion model of 
Zika virus  transmission.

\section{Preliminaries}\label{sec:preliminary}

In this section, we present some  properties of resolvent positive operators
and introduce the basic reproduction ratio $\R$  for abstract periodic systems. Then we address the continuity of $\R$ with respect to parameters.

\begin{definition}
	A square matrix is said to be cooperative if its off-diagonal elements are nonnegative, and nonnegative if all elements are nonnegative. A cooperative square matrix is said to be irreducible if it is not similar, via a permutation, to a block lower triangular matrix, and reducible if otherwise.
\end{definition}

\begin{lemma}\label{lem:reducible}
	{\sc (\cite[Section 2.3]{berman1994nonnegative})}
	Assume that $A$ is a nonnegative and reducible $n \times n$ matrix. Then there exists a permutation matrix $P$ such that 
	$$
	PAP^T=
	\left(
	\begin{matrix}
	A_{11} & 0 & \cdots & 0\\
	A_{12} & A_{22}& \cdots & 0\\
	\vdots & \vdots&\ddots &\vdots\\
	A_{\tilde{n}1} & A_{\tn 2} & \cdots & A_{\tn \tn}
	\end{matrix}
	\right),
	$$
	and $r(A)= \max_{1 \leq k \leq \tn} r(A_{kk})$.
\end{lemma}

\begin{definition}\label{def:resolvent}
	The spectral bound of a closed linear operator $A$ is defined as 
	$$
	s(A)=\sup \{\Realparts \lambda: \lambda \in \sigma(A)\}.
	$$
	A closed linear operator $A$ is said to be resolvent positive if the resolvent set of $A$ contains a ray $(\omega,+\infty)$ such that $(\lambda I -A)^{-1}$ is positive for all $\lambda > \omega$.
\end{definition}

\begin{lemma}\label{lem:R0:equiv:restatement}
	 {\sc (\cite[Theorem 3.5]{thieme2009spectral})}
	Assume that $(E,E_+)$ is an ordered Banach space with the positive cone $E_+$ being normal and reproducing. Assume that $B$ is a resolvent positive operator on $E$ with $s(B)<0$, $C$ is a positive operator on $E$ and $B+C$ is still a resolvent positive operator on $E$. Then $s(B+ C)$ has the same sign as $ r(-B^{-1} C)-1$.
\end{lemma}

\begin{lemma}\label{lem:R0:equiv}
	Assume that $(E,E_+)$ is an ordered Banach space with the positive cone $E_+$ being normal and reproducing. Assume that $B$ is a resolvent positive operator on $E$ with $s(B)<0$, $C$ is a positive operator on $E$, and $B+\frac{1}{\mu}C$ is still a resolvent positive operator on $E$ for any $\mu >0$. Then the following statements are valid:
	\begin{enumerate}
		\item[\rm (i)] For any $\mu >0$, $s(B+ \frac{1}{\mu} C)$ has the same sign as $ r(-B^{-1} C)-\mu$. 
		\item[\rm (ii)] If $ r(-B^{-1} C)>0$, $\mu=r(-B^{-1} C)$ is the unique solution of $s(B+ \frac{1}{\mu} C)=0$.
		\item[\rm (iii)] If $s(B + \frac{1}{\mu_0} C)\geq 0$
		for some $\mu_0>0$, then $ r(-B^{-1} C)>0$. Thus, if $ r(-B^{-1} C)=0$, then $s(B + \frac{1}{\mu} C)<0$ for all $\mu >0$.
		\item[\rm (iv)] If $s(B + \frac{1}{\mu} C)<0$ for all $\mu >0$, then $ r(-B^{-1} C)=0$.
	\end{enumerate}
\end{lemma}

\begin{proof}
	For any given $\mu>0$, Lemma \ref{lem:R0:equiv:restatement} implies that  $s(B+ \frac{1}{\mu} C)$ has the same sign as $ \frac{1}{\mu} r(-B^{-1} C)-1$. Thus, statement  (i) holds.  Statement  (ii) is a straightforward consequence of (i).  Statement  (i) also implies that $r(-B^{-1} C)\geq \mu_0> 0$, that is, statement  (iii) holds. Statement  (iv) can be derived directly from (ii).
\end{proof}

\begin{theorem}\label{thm:R0:conti}
	Let $(\Theta,d)$ be a metric space.
	Assume that $(E,E_+)$ is an ordered Banach space with the positive cone $E_+$ being normal and reproducing. For any $\theta \in \Theta$, let $B_{\theta}$ be a resolvent positive operator on $E$ with $s(B_{\theta})<0$, and let $C_{\theta}$ be a positive operator on $E$. Then $\lim\limits_{\theta \rightarrow \theta_0} r(-(B_{\theta})^{-1}C_{\theta})= r(-(B_{\theta_0})^{-1}C_{\theta_0})$
	for some $\theta_0 \in \Theta$ provided that 
	\begin{enumerate}
		\item[\rm (i)] For any $\mu >0$ and $\theta \in \Theta$, $B_{\theta} +\frac{1}{\mu} C_{\theta}$ is still a resolvent positive operator on $E$.
		\item[\rm (ii)] For any $ \mu >0$, $\lim\limits_{\theta \rightarrow \theta_0}s(B_{\theta}+ \frac{1}{\mu} C_{\theta})=s(B_{\theta_0}+ \frac{1}{\mu} C_{\theta_0})$.
	\end{enumerate}
\end{theorem}

\begin{proof}
	For convenience, we define $\R^{\theta}:=r(-(B_{\theta})^{-1}C_{\theta})$, $\forall \theta \in \Theta$.
	We divide the proof into two cases:
	
	{\it Case 1.} $\R^{\theta_0}>0$.
	For any given $\epsilon \in (0,\R^{\theta_0})$, 
	by Lemma \ref{lem:R0:equiv}(i), we obtain
	$$
	s(B_{\theta_0} + \frac{1}{\R^{\theta_0}-\epsilon} C_{\theta_0})>0,
	\text{ and }
	s(B_{\theta_0} + \frac{1}{\R^{\theta_0} +\epsilon} C_{\theta_0})<0.
	$$
	Thanks to assumption (ii), there exists $\delta>0$ such that if $d(\theta,\theta_0) \leq \delta$, then $s(B_{\theta} + \frac{1}{\R^{\theta_0}-\epsilon}C_{\theta})>0$ and $s(B_{\theta} + \frac{1}{\R^{\theta_0}+\epsilon} C_{\theta})<0$. By 
	Lemma \ref{lem:R0:equiv} (i) again, it follows that
	$$
	0<\R^{\theta_0}- \epsilon < \R^{\theta} < \R^{\theta_0}+ \epsilon$$
	provided $d(\theta,\theta_0) \leq \delta$. This shows that $\lim\limits_{\theta \rightarrow \theta_0} \R^{\theta} = \R^{\theta_0}$.
	
	{\it Case 2.} $\R^{\theta_0}=0$.
	Let $$\Lambda:=\{\theta \in \Theta: \R^{\theta}>0\}.$$ Lemma \ref{lem:R0:equiv}(ii) implies that $s(B_{\theta} + \frac{1}{\R^{\theta}}C_{\theta})=0$, $\forall \theta \in \Lambda$. For any given $\epsilon>0$, using Lemma \ref{lem:R0:equiv}(iii) and $ \R^{\theta_0}=0$, we have 
	$
	s(B_{\theta_0} + \frac{1}{\epsilon} C_{\theta_0})<0
	$. By assumption (ii), there exists $\delta>0$ such that 
	$s(B_{\theta} + \frac{1}{\epsilon} C_{\theta})<0$ for all $\theta \in \Lambda$ with $d(\theta,\theta_0) \leq \delta$. Thanks to Lemma \ref{lem:R0:equiv}(i), we conclude that $\R^{\theta}< \epsilon$ provided $\theta \in \Lambda$  and $d(\theta,\theta_0) \leq \delta$. Thus, $\lim\limits_{\theta \rightarrow \theta_0} \R^{\theta} = 0$.
\end{proof}

\begin{theorem}\label{thm:R0:conti:0}
	Let $(\Theta,d)$ be a metric space and let $\theta_0 \in \Theta$  be given.
	Assume that $(E,E_+)$ is an ordered Banach space with the positive cone $E_+$ being normal and reproducing. For any $\theta \in \Theta \setminus \{ \theta_0 \}$, let $B_{\theta}$ be a resolvent positive operator on $E$ with $s(B_{\theta})<0$, and let $C_{\theta}$ be a positive operator on $E$. Then $\lim\limits_{\theta \rightarrow \theta_0} r(-(B_{\theta})^{-1}C_{\theta})= 0$ for some $\theta_0 \in \Theta$ provided that 
	\begin{enumerate}
		\item[\rm (i)] For any $\mu >0$ and 
		$\theta \in \Theta\setminus \{ \theta_0\}$, $B_{\theta} +\frac{1}{\mu} C_{\theta}$ is still a resolvent positive operator on $E$.
		\item[\rm (ii)] For any $ \mu >0$, $\limsup\limits_{\theta \rightarrow \theta_0} s(B_{\theta}+ \frac{1}{\mu} C_{\theta}) <0$.
	\end{enumerate}
\end{theorem}
\begin{proof}
	We prove this theorem by modifying the arguments for the case where $\R^{\theta_0}
	=0$ in the proof of Theorem \ref{thm:R0:conti}. For convenience, we define $\R^{\theta}:=r(-(B_{\theta})^{-1}C_{\theta})$, $\forall \theta \in \Theta$.
	Let $$\Lambda:=\{\theta \in \Theta: \R^{\theta}>0\}.$$ 
	Then  Lemma \ref{lem:R0:equiv}(ii) implies that $s(B_{\theta} + \frac{1}{\R^{\theta}}C_{\theta})=0$, $\forall \theta \in \Lambda$. For any given $\epsilon>0$, by assumption (ii), there exists $\delta >0$ such that 
	$s(B_{\theta} + \frac{1}{\epsilon} C_{\theta})<0$ 
	for all $\theta \in \Lambda$ with $d(\theta,\theta_0) \leq \delta$. Thanks to Lemma \ref{lem:R0:equiv}(i), we conclude that $\R^{\theta}< \epsilon$ provided $\theta \in \Lambda$ and $d(\theta,\theta_0) \leq \delta$. Thus, $\lim\limits_{\theta \rightarrow \theta_0} \R^{\theta} = 0$. 
\end{proof}


Let $T>0$ be a given real number. Next we introduce the concept of 
periodic evolution families in order to study periodic evolution systems.

\begin{definition}
	A family of bounded linear operators $\varUpsilon(t,s)$ on a Banach space $E$, $t,s \in \Real$ with $t\geq s$, is called  a $T$-periodic evolution family provided that 
	$$
	\varUpsilon(s,s)=I,\quad
	\varUpsilon(t,r)\varUpsilon(r,s)=\varUpsilon(t,s),\quad 
	\varUpsilon(t+T,s+T)=\varUpsilon(t,s),
	$$
	for all $t,s,r \in \Real$ with $t\geq r \geq s$, and for each $e\in E$, $\varUpsilon(t,s)e$ is a continuous function of $(t,s)$, $t \geq s$. The exponential growth bound of evolution family $\{\varUpsilon(t,s): t\geq s\}$ is defined as 
	\begin{equation}\label{equ:E_G_B}
	\omega(\varUpsilon) = \inf \{\tilde{\omega} \in \Real : \exists M \geq 1: \forall t,s \in \Real,~t \geq s: \Vert \varUpsilon(t,s)\Vert \leq M e^{\tilde{\omega }(t-s)} \}.
	\end{equation}
\end{definition}

\begin{lemma}{ \sc (\cite[Propostion A.2]{thieme2009spectral})}
	\label{lem:w_theta:equ}
	Let $E$ be a Banach space and let $\{ \varUpsilon(t,s): t\geq s \}$ be a $T$-periodic evolution family on a Banach space $E$. Then 
	$\omega(\varUpsilon)=\frac{\ln r(\varUpsilon(T,0))}{T}=\frac{\ln r(\varUpsilon(T+\tau,\tau))}{T},~ \forall \tau\in [0,T]$.
\end{lemma}

Let $Y$ be a Banach space equipped the norm $\Vert \cdot \Vert_{Y}$ with the positive cone $Y_+$ being normal and reproducing. It is easy to see that 
$$
\bY:=\{ \bmu \in C(\Real,Y): \bmu(t)=\bmu(t+T), ~t \in \Real \}
$$
is an ordered  Banach space with the positive cone
$$
\bY_+:=\{ \bmu \in C(\Real,Y_+): \bmu(t)=\bmu(t+T), ~t \in \Real \}
$$
and the maximum norm $\Vert \cdot \Vert_{\bY} $.

Let $\{\Phi(t,s): t\geq s \}$ be a $T$-periodic evolution family on $Y$ and let $(F(t))_{t \in \Real}$ be a family of bounded linear operators on $Y$ such that
$F(t+T)=F(t)$ for all $t \in \Real$. We assume that 
\begin{enumerate}
	\item[(A1)] $F(t)$ is positive for each $t \in \Real$, and $F(t)$ is strongly continuous in $t \in \Real$ in the sense that for any $y \in Y$, $t_0 \in \Real$, $\Vert F(t) y- F(t_0) y \Vert_Y \rightarrow 0$ as $t \rightarrow t_0^+$.
	
	\item[(A2)]  $\Phi(t,s)$ is positive for any $t,s \in \Real$ with $t \geq s$, and $\omega(\Phi)<0$.
\end{enumerate}

For each $\mu >0$, let $\{\Psi_{\mu}(t,s): t\geq s \}$ be the uniquely determined family on $Y$ which solves 
\begin{equation}\label{equ:Phi_Psi}
\Psi_{\mu}(t,s)\phi - \Phi(t,s) \phi =
\frac{1}{\mu} \int_{s}^{t}\Psi_{\mu}(t,\tau)F(\tau) \Phi(\tau,s) \phi \D \tau= \frac{1}{\mu} \int_{s}^{t}\Phi(t,\tau) F(\tau) \Psi_{\mu}(\tau,s) \phi \D \tau.
\end{equation}
By \cite[Corollary 5]{martin1990abstract}, $\Psi_{\mu}(t,s)\phi$ is positive for any $t,s \in \Real$ with $t \geq s$ and $\mu >0$.

In mathematical epidemiology, $F(t)$ represents the infection operator at time $t$ and
$\Phi (t,s)$ is generated by solutions of a periodic  internal evolution system of the populations at some infected compartments.  As such, $F(t)u(t)$ denotes the distribution of newly infected individuals at time $t$; $\Phi (t,t-s) F(t-s)u (t-s)$ is the distribution of those infected individuals who were newly infected at time $t-s$ and remain in the infected compartments at time $t$;
and $F(t)\Phi(t,t-s) u (t-s)$ is  the distribution of the individuals newly infected at time $t$ by those infected individuals who 
were introduced at time $t-s$.  We introduce two bounded linear positive operators $L:\bY \rightarrow \bY $ and $\hat{L}:\bY \rightarrow \bY $ by
\begin{equation}\label{equ:L}
[L u](t): = \int_{0}^{+\infty} \Phi(t,t-s)F(t-s)u (t-s)\D s, ~t \in \Real,~ u \in \bY,
\end{equation}
and
\begin{equation}\label{equ:hatL}
[\hat{L} u](t): =  \int_{0}^{+\infty} F(t) \Phi(t,t-s) u (t-s)\D s, ~t \in \Real,~ u \in \bY.
\end{equation}
Let $L_1$ and $L_2$ be two bounded linear positive operators from $\bY$ to $\bY$ defined by
$$
[L_1 u](t): = \int_{0}^{+\infty} \Phi(t,t-s) u (t-s)\D s,~ [L_2 u](t):=F(t)u(t), ~t \in \Real,~ u \in \bY.
$$
It then follows that $L=L_1 \circ L_2$ and $\hat{L}=L_2 \circ L_1$, and hence, $L$ and $\hat{L}$ have the same spectral raidus.
Following \cite{liang2019basic,wang2008threshold,wang2012basic,thieme2009spectral}, we define the spectral radius of $L$ and $\hat{L}$ on $\bY$ as the basic reproduction ratio, that is, $$\R=r(L)=r(\hat{L}).$$
According to Lemmas \ref{lem:R0:equiv} and  \ref{lem:w_theta:equ} and  \cite[Section 5]{thieme2009spectral} (see also \cite{liang2019basic}), 
we then have the following two results.
\begin{theorem}\label{thm:R0:periodic}
	Assume that {\rm (A1)} and {\rm (A2)} hold. Then the following statements are valid:
	\begin{enumerate}
		\item[\rm (i)]  $\R -1$ has the same sign as $\omega(\Psi_1)$.
		\item[\rm (ii)] If $\R>0$, then $\mu=\R$ is the unique solution of $\omega(\Psi_{\mu})=0$. 
	\end{enumerate}
\end{theorem}

\begin{corollary}\label{cor:periodic:R0}
	Assume that {\rm (A1)} and {\rm (A2)} hold. Then the following statements are valid:
	\begin{enumerate}
		\item[\rm (i)] For any $\mu >0$, $\R -\mu$ has the same sign as
		$\omega(\Psi_{\mu})$.
		\item[\rm (ii)] If $\omega(\Psi_{\mu_0})\geq 0$ for some  $\mu_0>0$, then $\R >0$. Thus, if $\R=0$, then $\omega(\Psi_{\mu})<0$ for all $\mu >0$. Conversely, if $\omega(\Psi_{\mu})<0$ for all $\mu >0$, then $\R=0$.
	\end{enumerate}
\end{corollary}

Now we are in a position to prove the continuity of the basic reproduction ratio.

\begin{theorem}\label{thm:R0:periodic:conti}             
	Let $(\Theta,d)$ be a metric space. For any $\theta \in \Theta$, let $\{\Phi^{\theta}(t,s): t\geq s \}$ be a $T$-periodic evolution family on $Y$, and let $(F^{\theta}(t))_{t \in \Real}$ be a family of bounded linear operators on $Y$ such that $F^{\theta}(t+T)=F^{\theta}(t)$ for all $t \in \Real$. 
	Assume that for each $\theta \in \Theta$, 
	{\rm (A1)} and {\rm (A2)} hold with $\Phi$ and $F$ replaced by $\Phi^{\theta}$ and $F^{\theta}$, respectively. For any $\theta \in \Theta$, let $\Psi_{\mu}^{\theta}$, $L^{\theta}$ and $\R^{\theta}$ be defined similarly to $\Psi_{\mu}$, $L$ and $\R$. Then $\lim\limits_{\theta \rightarrow \theta_0} \R^{\theta} = \R^{\theta_0}$ for some $\theta_0 \in \Theta$ provided that $\lim\limits_{\theta \rightarrow \theta_0}\omega(\Psi_{\mu}^{\theta}) =\omega(\Psi_{\mu}^{\theta_0})$ for all $\mu \in (0,\infty)$. 
\end{theorem}

\begin{proof}
	Indeed, this theorem is a straightforward consequence of Theorem \ref{thm:R0:conti} and \cite[Section 5]{thieme2009spectral}. For reader's convenience, below we provide an elementary proof by modifying the arguments for Theorem \ref{thm:R0:conti}. It suffices to consider two cases:
	
	{\it Case 1.} $\R^{\theta_0}>0$.
	For any given $\epsilon \in (0,\R^{\theta_0})$, 
	by Corollary \ref{cor:periodic:R0}(i), we obtain
	$$
	\omega(\Psi_{\R^{\theta_0}-\epsilon}^{\theta_0})>0,
	\text{ and }
	\omega(\Psi_{\R^{\theta_0}+\epsilon}^{\theta_0})<0.
	$$
	Thanks to the assumption $\lim\limits_{\theta \rightarrow \theta_0}\omega(\Psi_{\mu}^{\theta}) =\omega(\Psi_{\mu}^{\theta_0})$ for all $\mu \in (0,\infty)$, there exists $\delta>0$ such that if $d(\theta,\theta_0) \leq \delta$, then $\omega(\Psi_{\R^{\theta_0}-\epsilon}^{\theta})>0$ and $\omega(\Psi_{\R^{\theta_0}+\epsilon}^{\theta})<0$. By Corollary \ref{cor:periodic:R0}(i) again, it follows that
	$$
	0<\R^{\theta_0}- \epsilon < \R^{\theta} < \R^{\theta_0}+ \epsilon$$
	provided $d(\theta,\theta_0) \leq \delta$. This shows that $\lim\limits_{\theta \rightarrow \theta_0} \R^{\theta} = \R^{\theta_0}$.
	
	{\it Case 2.} $\R^{\theta_0}=0$.
	Let $$\Lambda:=\{\theta \in \Theta: \R^{\theta}>0\}.$$  Then Corollary \ref{cor:periodic:R0}(i) implies that $\omega(\Psi_{\R^{\theta}}^{\theta})=0$, $\forall \theta \in \Lambda$. For any given $\epsilon>0$, using Corollary \ref{cor:periodic:R0}(ii) and $ \R^{\theta_0}=0$, we have 
	$
	\omega(\Psi_{\epsilon}^{\theta_0})<0
	$. By assumption $\lim\limits_{\theta \rightarrow \theta_0}\omega(\Psi_{\mu}^{\theta}) =\omega(\Psi_{\mu}^{\theta_0})$ for all $\mu \in (0,\infty)$ again, there exists $\delta>0$ such that $
	\omega(\Psi_{\epsilon}^{\theta})<0
	$ for all $\theta \in \Lambda$ with $d(\theta,\theta_0) \leq \delta$. In view of  Corollary \ref{cor:periodic:R0}(i), we conclude that $\R^{\theta}< \epsilon$ provided $\theta \in \Lambda$ and $d(\theta,\theta_0) \leq \delta$. Thus, $\lim\limits_{\theta \rightarrow \theta_0} \R^{\theta} = 0$.
\end{proof}

\begin{theorem}\label{thm:R0:periodic:conti:0}
	Let $(\Theta,d)$ be a metric space and let $\theta_0 \in \Theta$  be given. For any $\theta \in \Theta\setminus \{ \theta_0\}$, let $\{\Phi^{\theta}(t,s): t\geq s \}$ be a $T$-periodic evolution family on $Y$, and let $(F^{\theta}(t))_{t \in \Real}$ be a  family of bounded linear operators on $Y$ such that $F^{\theta}(t+T)=F^{\theta}(t)$
	for all $t \in \Real$.  Assume that for each $\theta \in \Theta\setminus \{ \theta_0\}$,
	{\rm (A1)} and {\rm (A2)} hold with $\Phi$ and $F$ replaced by $\Phi^{\theta}$ and $F^{\theta}$,  respectively, and  let $\Psi_{\mu}^{\theta}$, $L^{\theta}$ and $\R^{\theta}$ be defined similarly to $\Psi_{\mu}$, $L$ and $\R$. Then $\lim\limits_{\theta \rightarrow \theta_0} \R^{\theta} = 0$ provided that $\limsup\limits_{\theta \rightarrow \theta_0}\omega(\Psi_{\mu}^{\theta}) <0$ for all $\mu \in (0,\infty)$. 
\end{theorem}

\begin{proof}
	The proof is motivated by the arguments for Theorem \ref{thm:R0:conti:0}. Let $$\Lambda:=\{\theta \in \Theta: \R^{\theta}>0, ~\theta \neq \theta_0\}.$$ 
	Then Corollary \ref{cor:periodic:R0}(i) implies that $\omega(\Psi_{\R^{\theta}}^{\theta})=0$, $\forall \theta \in \Lambda$. For any given $\epsilon>0$, by assumption $\limsup\limits_{\theta \rightarrow \theta_0}\omega(\Psi_{\mu}^{\theta}) <0$ for all $\mu \in (0,\infty)$, there exists $\delta>0$ such that 
	$\omega(\Psi_{\epsilon}^{\theta})<0$ for all $\theta \in \Lambda$ with $d(\theta,\theta_0) \leq \delta$. 
	Thanks to Corollary \ref{cor:periodic:R0}(i) again, we conclude that $\R^{\theta}< \epsilon$ provided $\theta \in \Lambda$ and $d(\theta,\theta_0) \leq \delta$. Thus, $\lim\limits_{\theta \rightarrow \theta_0} \R^{\theta} = 0$.
\end{proof}

\section{The principal eigenvalue}\label{sec:AB}
In this section, we  investigate the asymptotic behavior of the principal eigenvalue for periodic cooperative reaction-diffusion systems as the diffusion coefficients go to infinity.  We start with some notations and 
the related known results. 

Consider the following periodic parabolic linear system:
\begin{equation}\label{equ:sys:main}
\begin{cases}
\frac{\partial \bmv}{\partial t}-  \bmkp \cL(x,t) \bmv - \cM(x,t) \bmv =0, &x \in \Omega, ~t>0,\\
\cB \bmv =0, & x \in \partial\Omega,~ t>0,
\end{cases}
\end{equation}
and the associated eigenvalue problem:
\begin{equation}\label{equ:eig:main}
\begin{cases}
\frac{\partial \bmu}{\partial t}= \bmkp \cL(x,t) \bmu + \cM(x,t) \bmu +\lambda \bmu, &(x,t) \in \Omega \times \Real,\\
\cB \bmu =0, &(x,t) \in \partial\Omega \times \Real.
\end{cases}
\end{equation}
Here $\bmv=(v_1,v_2,\cdots,v_n)^T$ and $\bmu=(u_1,u_2,\cdots,u_n)^T $; $\bmkp={\rm diag}(\kappa_1,\kappa_2,\cdots,\kappa_n)$ with $\kappa_i>0$; 
$\cL(x,t)={\rm diag}(\cL_1,\cL_2,\cdots,\cL_n)(x,t)$ is of the  divergence form
\begin{equation}\label{equ:sym:diver}
\cL_{i}(x,t) u_i: = \sum_{p,q=1}^{N} \frac{\partial}{\partial x_q}\left( a_{pq}^i (x,t) \frac{\partial}{\partial x_p } u_i\right), ~1 \leq i \leq n,
\end{equation}
where $a_{pq}^{i} \in C^{1,\alpha}(\oOmega \times \Real)$ and $\frac{\partial}{\partial x_q}a_{pq}^{i} \in C^{\alpha}(\oOmega \times \Real) $ with $a_{pq}^{i}(x,t)=a_{pq}^{i}(x,t+T)$ and $a_{pq}^{i}(x,t)=a_{qp}^{i}(x,t)$, $\forall 1 \leq p,q \leq N,~ 1 \leq i \leq n,~ (x,t) \in \oOmega \times \Real$ ($0<\alpha<1$); $ \underline{a} \sum_{p=1}^{N} \xi_p^2 \leq  \sum_{p,q=1}^{N} a_{pq}^{i}(x,t) \xi_p \xi_q \leq  \overline{a} \sum_{p=1}^{N} \xi_p^2$, $\forall (x,t) \in \oOmega \times \Real$ for some $\overline{a} \geq \underline{a}>0$; 
$\cM=(m_{ij})_{n\times n}$ is a H\"{o}lder continuous $n\times n$ matrix-valued function of $(x,t) \in \overline{\Omega}\times \Real$ with $\cM(x,t)=\cM(x,t+T)$ in the sense that each $m_{ij}$ is H\"{o}lder continuous on $\overline{\Omega}\times \Real$ with $m_{ij}(x,t)=m_{ij}(x,t+T)$, $\forall 1 \leq i,j \leq n$, $(x,t) \in \oOmega \times \Real$; $\cB={\rm diag}(\cB_1,\cB_2,\cdots,\cB_n)$ denotes  a boundary operator such that for each $i$, $\cB_i$ represents either the Dirichlet boundary condition
\begin{equation}\label{equ:Dirichlet}
\cB_i u_i := u_i, ~ \text{ on }\partial \Omega,
\end{equation}
or the Neumann boundary condition
\begin{equation}\label{equ:Neumann}
\cB_i u_i := \sum_{p,q=1}^{N} a_{pq}^{i}\frac{\partial u_i}{\partial x_p}\cos(\bmnu,\bm{x}_q), ~\text{ on }\partial \Omega,
\end{equation} 
or the Robin boundary condition
\begin{equation}\label{equ:Robin}
\cB_i u_i := \sum_{p,q=1}^{N} a_{pq}^{i}\frac{\partial u_i}{\partial x_p}\cos(\bm{\nu},\bm{x}_q)+ b_i u_i, ~ \text{ on }\partial \Omega,
\end{equation}
where $\bmnu$ is the outward unit normal vector of the boundary $\partial \Omega$; $\bm{x}_q$ is a unit $N$-dimensional vector whose only $q$-th component is nonzero; $b_i$ is a H\"{o}lder continuous function of $(x,t)\in \oOmega$ with $b_i(x,t) > 0$ and $b_i(x,t)=b_i(x,t+T)$, $\forall (x,t) \in \oOmega \times \Real$. We say $\cB$ satisfies the Dirichlet (Neumann and Robin) boundary condition if all $\cB_i$ satisfy the Dirichlet (Neumann and Robin) boundary conditions given in \eqref{equ:Dirichlet} (\eqref{equ:Neumann} and \eqref{equ:Robin}). Throughout this paper, we always assume that $\cB$ satisfies the Dirichlet or Neumann or Robin boundary condition. We further assume that
\begin{itemize}
	\item[(M)] $\cM(x,t)$ is cooperative for all $(x,t) \in \oOmega \times \Real$.
\end{itemize}

\begin{definition}\label{def:principal}
$\lambda^*$ is called the principal eigenvalue of \eqref{equ:eig:main} if it is a real eigenvalue with a nonnegative eigenfunction and the real parts
of all other eigenvalues are greater than $\lambda^*$.
\end{definition}

We  use $X$ and $X_+$ to represent $X_0:=C_0(\oOmega,\Real^n)$ and $X_{0,+}:=C_0(\oOmega,\Real^n_+)$ if $\cB$ satisfies the Dirichlet boundary condition, and $X_1:=C(\oOmega,\Real^n)$ and $X_{1,+}:=C(\oOmega,\Real^n_+)$ if $\cB$ satisfies the Neumann or Robin boundary condition. It is easy to see that $(X, X_{+})$  is an ordered Banach space with the maximum norm $\Vert \bm{\phi} \Vert_{X}= \max_{1 \leq i \leq n} \max_{x \in \oOmega} \vert \phi_i(x) \vert$
for $\bm{\phi}=(\phi_1,\phi_2,\cdots,\phi_n)^T$. We equip the space of $T$-periodic functions 
$$
\bX:=\{ \bmu \in C( \Real,X): \bmu(t)=\bmu(t+T), ~t \in \Real\}
$$
with the positive cone
$$
\bX_+:=\{ \bmu \in C( \Real,X_+): \bmu(t)=\bmu(t+T),~ t \in \Real\}
$$
and the norm $\Vert \bmu \Vert_{\bX} = \max_{0 \leq t \leq T} \Vert \bmu (t) \Vert_{X}$. Then $(\bX, \bX_+)$ is
an ordered Banach space. For any $\bmu \in \bX$, we use $\bmu(x,t)$ to represent $[\bmu(t)](x)$, $(x,t)\in \oOmega \times \Real$ unambiguously. In the case where $X=X_1$, $\bX$ can  be identified with the  Banach space
$$
\{ \bmu \in C(\oOmega \times \Real,\Real^n): \bmu(x,t)=\bmu(x,t+T), ~ (x,t)\in \oOmega \times \Real\}.
$$
In the case where $X=X_0$, $\bX$ can be  identified with the  Banach space
$$
\{ \bmu \in C(\oOmega \times \Real,\Real^n): \bmu(x,t)=\bmu(x,t+T), ~ (x,t)\in \oOmega \times \Real, \text{ and } \bmu(x,t)=0,~(x,t) \in \partial \Omega \times \Real\}.
$$

According to \cite[Section 13]{hess1991periodic} (see also  \cite[Chapters 5 and 6]{lunardi1995analytic}), system \eqref{equ:sys:main} admits a unique evolution family $\{\bU(t,s): t \geq s \}$ on $X$. By the celebrated Krein-Rutman theorem (see, e.g., \cite[Theorem 19.2]{deimling1985nonlinear}) and Lemma \ref{lem:w_theta:equ}, we have the following result (see, e.g., \cite[Theorem 2.7]{liang2017principal},  \cite[Theorem 1.3]{bai2020asymptotic}  and \cite[Theorem 3.4]{anton1992strong}).
\begin{theorem}\label{thm:principal:existence}
	Assume that {\rm (M)} holds, and $\cB$ satisfies the Dirichlet or Neumann or Robin boundary condition.
	Then the eigenvalue problem \eqref{equ:eig:main} admits the principal eigenvalue $\lambda^*= - \omega (\bU)= -\frac{\ln r(\bU(T,0))}{T}$.
\end{theorem}

For each $x \in \oOmega$, let $\{O_x(t,s): t \geq s\}$ be the evolution family on $\Real^n$ of the linear ODE system $$\frac{\D \bmv}{\D t}=\cM(x,t) \bmv,~ t>0.$$ 
Let $\{\bO(t,s): t \geq s\}$ be the evolution family on $X$ of $\frac{\D \bmv}{\D t}=\cM(x,t) \bmv, ~x \in \oOmega,~t>0$. 
\begin{proposition}\label{prop:eta}
	Assume that {\rm (M)} holds. Then $\omega(\bO)=\max_{ x \in \oOmega} \omega(O_x)=\max_{ x \in \oOmega} \frac{\ln r(O_x(T,0))}{T}$.
\end{proposition}
\begin{proof}
	In the case where $X=X_1$, it follows from \cite[Proposition 2.7]{liang2019principal} that $\sigma(\bO(T,0))= \cup_{x \in \oOmega}
	\sigma(O_x (T,0))$.
	By the perturbation theory of matrix (see, e.g., \cite[Section II.5.1]{kato1976perturbation}), $r(O_x (T,0))$ is continuous with respect to $x \in \oOmega$. This implies that $\max_{ x \in \oOmega} r(O_x (T,0))$ exists, and hence, $r(\bO(T,0)) = \max_{ x \in \oOmega} r(O_x (T,0))$. Now the desired conclusion follows from Lemma \ref{lem:w_theta:equ}.
	
	In the case where $X=X_0$, it suffices to show $\sigma(\bO(T,0))= \cup_{x \in \oOmega} \sigma(O_x (T,0))$. By arguments similar to those in 
	Step 1 of the proof of \cite[Proposition 2.7]{liang2019principal}, we have $\sigma(\bO(T,0)) \supset \cup_{x \in \Omega} \sigma(O_x (T,0))$. Since $\sigma(\bO(T,0)) $ is a closed set (see, e.g., \cite[Theorem VI.5]{reed1980methods}), it then follows that $\sigma(\bO(T,0)) \supset \overline{\cup_{x \in \Omega} \sigma(O_x (T,0))}$. 
	We further have the following claim.
	
	{\it Claim:} $ \overline{\cup_{x \in \Omega} \sigma(O_x (T,0))}=\cup_{x \in \oOmega} \sigma(O_x (T,0))$. 
	
	It is easy to see that $\overline{\cup_{x \in \Omega} \sigma(O_x (T,0))}\supset\cup_{x \in \Omega} \sigma(O_x (T,0))$. By the perturbation theory of matrix (see, e.g., \cite[Section II.5.1]{kato1976perturbation}), for any $x_0 \in \partial \Omega$, $x_n \in \Omega$ with $ x_n \rightarrow x_0$ and $\lambda_0 \in \sigma(O_{x_0}(T,0))$, there exists a sequence $\lambda_n \in \sigma(O_{x_n}(T,0))$ such that $\lambda_n \rightarrow \lambda_0$ as $ x_n \rightarrow x_0$. This implies that $\lambda_0\in \overline{\cup_{x \in \Omega} \sigma(O_x (T,0))}$, and hence, $\overline{\cup_{x \in \Omega} \sigma(O_x (T,0))}\supset\cup_{x \in \oOmega} \sigma(O_x (T,0))$. Similarly, we can prove $\cup_{x \in \oOmega} \sigma(O_x (T,0))$ is a closed set, which yields that $\overline{\cup_{x \in \Omega} \sigma(O_x (T,0))}\subset\cup_{x \in \oOmega} \sigma(O_x (T,0))$. The claim is proved.
	
	Thus, $\sigma(\bO(T,0)) \supset \cup_{x \in \oOmega} \sigma(O_x (T,0))$. Repeating the arguments in Step 2 of \cite[Proposition 2.7]{liang2019principal}, we conclude that $\sigma(\bO(T,0)) = \cup_{x \in \Omega} \sigma(O_x (T,0))$.
\end{proof}
Let  $\eta:=\max_{x \in \oOmega}\omega(O_x)$. Note that $- \omega(O_x)$ is the principal eigenvalue of $\frac{\D \bmu}{\D t}=\cM(x,t) \bmu + \lambda \bmu$ for any $x \in \oOmega$, and $-\eta= \min_{x \in \oOmega} -\omega (O_x)$.  

\begin{theorem}\label{thm:0}
	{\sc (\cite[Theorem 1.5]{bai2020asymptotic})}
	Assume that {\rm (M)} holds.
	Let $\lambda_{\bmkp}^*$ be the principal eigenvalue of the eigenvalue problem \eqref{equ:eig:main} subject to the Dirichlet or Neumann or Robin boundary condition. Then we have 
	$$\lim_{\max_{1 \leq i \leq n} \kappa_i \rightarrow 0} \lambda_{\bmkp}^* = - \eta.$$
\end{theorem}
Next we turn to the asymptotic behavior of the principal eigenvalue with large diffusion coefficients.
It is easy to see the following observation is valid.
\begin{lemma}\label{lem:transfer}
	Assume that {\rm (M)} holds, and $\cB$ satisfies the Dirichlet or Neumann or Robin boundary condition.
	Let $\lambda^{*}$ be the principal eigenvalue of the eigenvalue problem \eqref{equ:eig:main}. Then $\lambda^* -c_0$ is the principal eigenvalue of the eigenvalue problem \eqref{equ:eig:main} with $\cM(x,t)$ replaced by $\cM(x,t) +c_0 I$, where $I$ is an identity matrix.
\end{lemma}
 If (M) holds, we can choose a sufficient large $c_0>0$ such that for each $(x,t) \in \oOmega \times \Real$, $\cM(x,t)+c_0 I$ is a nonnegative matrix.
In view of Lemma \ref{lem:transfer}, without loss of generality, in the remaining part of this section we assume that
 \begin{itemize}
 	\item[(SM)] $\cM(x,t)$ is nonnegative for all $(x,t) \in \oOmega \times \Real$.
 \end{itemize}

The following result provides an elemental  tool in our proof of  Theorems \ref{thm:infty:Neumann} and \ref{thm:infty:D_R}.
\begin{lemma}\label{lem:lower_upper}
	Assume that {\rm (SM)} holds.
	Let $\lambda^*$ be the principal eigenvalue of the eigenvalue problem \eqref{equ:eig:main} subject to the Dirichlet or Neumann or Robin boundary condition. Let
	$$
	\hat{\bX}:=\{\bmu\in C^{2,1}(\oOmega \times \Real,\Real^n): \bmu(x,t)=\bmu(x,t+T)\}.
	$$ Then the following statements are valid:
	\begin{enumerate}
		\item[\rm (i)] $\lambda^* \leq \lambda_0 $ holds provided that  there are $\lambda_0$ and $\hat{\bmu} \in \hat{\bX} \cap \bX_+$ with $\hat{\bmu} \neq 0$ such that
		\begin{equation}\label{equ:lower}
		\begin{cases}
		\frac{\partial \hat{\bmu}}{\partial t}- \bmkp \cL(x,t) \hat{\bmu} - \cM(x,t) \hat{\bmu} \leq\lambda_0 \hat{\bmu}, &(x,t) \in \Omega \times \Real,\\
		\cB \hat{\bmu} =0, &(x,t) \in \partial\Omega \times \Real.
		\end{cases}
		\end{equation}
		\item[\rm (ii)] $\lambda^* \geq \lambda_0 $ holds 
		provided that there exists $x_0 \in \oOmega$ such that $\cM(x_0,t)$ is irreducible for all $t \in \Real$ and there are $\lambda_0$ and $\hat{\bmu} \in \hat{\bX} \cap \bX_+$ with $\hat{\bmu} \neq 0$ such that
		\begin{equation}\label{equ:upper}
		\begin{cases}
		\frac{\partial \hat{\bmu}}{\partial t}- \bmkp \cL(x,t) \hat{\bmu} - \cM(x,t) \hat{\bmu} \geq\lambda_0 \hat{\bmu}, &(x,t) \in \Omega \times \Real,\\
		\cB \hat{\bmu} =0, & (x,t) \in \partial\Omega \times \Real.
		\end{cases}
		\end{equation}
	\end{enumerate}
\end{lemma} 

\begin{proof}
	Without loss of generality, we only consider the case where $\cB$ satisfies the Neumann boundary condition, since the conclusions can be proved for other boundary conditions in a similar way. 

	(i) Without loss of generality, we assume that $\bm{\phi} :=\hat{\bmu}(0) \neq 0$. Otherwise, we can replace $t$ by $t+t_0$ if $ \hat{\bmu}(t_0) \neq 0$. Since $\{e^{\lambda_0(t-s) }\bU(t,s): t \geq s \}$ is the evolution family of
	$$
	\begin{cases}
	\frac{\partial \bmv}{\partial t}- \bmkp \cL(x,t) \bmv - \cM(x,t) \bmv =\lambda_0 \bmv, &x\in \Omega,~ t>0,\\
	\cB \bmv =0, &x \in \partial\Omega,~ t>0,
	\end{cases}
	$$
	it follows from \eqref{equ:lower} and the comparison principle that $e^{\lambda_0 T} \bU(T,0) \bm{\phi} \geq \bm{\phi}$. 
	Therefore, Gelfand’s formula (see, e.g., \cite[Theorem VI.6]{reed1980methods}) yields that $r (e^{\lambda_0 T} \bU(T,0)) \geq 1$. Moreover, $r (e^{\lambda^* T} \bU(T,0)) = 1$ by Theorem \ref{thm:principal:existence}. Thus, we have $\lambda_0 \geq \lambda^*$.
	
	(ii) Without loss of generality, we assume that $\bm{\phi} :=\hat{\bmu}(0) \neq 0$. Otherwise, we can replace $t$ by $t+t_0$ if $ \hat{\bmu}(t_0) \neq 0$. By \eqref{equ:upper}, we obtain $ e^{\lambda_0 T} \bU(T,0)) \bm{\phi} \leq \bm{\phi}$. Since $\cM(x_0,t)$ is irreducible for all $t \in \Real$, $\bU(T,0)$ is eventually strongly positive (see, e.g., \cite[Theorem 2.7]{liang2017principal}).  By the compactness of  $\bU(T,0)$
	and the Krein-Rutman theorem (see, e.g., \cite[Section 7]{hess1991periodic}),  it follows that $r (e^{\lambda_0 T} \bU(T,0)) \leq 1$. Thus, $\lambda^* \geq \lambda_0 $.
\end{proof}

By the standard comparison arguments, we have the following result.
\begin{proposition}\label{prop:eig:com}
	Assume that {\rm (SM)} holds.
	Let $\lambda^{\cD}$, $\lambda^{\cN}$ and $\lambda^{\cR}$ be the principal eigenvalue of the eigenvalue problem \eqref{equ:eig:main} subject to the Dirichlet, Neumann and Robin boundary conditions, repsectively.
	Let $\Omega_0 \subset \Omega$ with smooth boundary, we use $\lambda^{\cD_0}$ to denote the principal eigenvalue of the eigenvalue problem \eqref{equ:eig:main} subject to Dirichlet boundary condition with $\Omega$ replaced by $\Omega_0$. Let $ \hat{\cM}=(\hat{m}_{ij})_{n\times n}$ be a H\"{o}lder continuous $n\times n$ matrix-valued function of $(x,t) \in \overline{\Omega}\times \Real$ with $\hat{\cM}(x,t)=\hat{\cM}(x,t+T)$ and $\hat{m}_{ij}(x,t) \geq m_{ij}(x,t) \geq 0$, $1 \leq i,j \leq n$. We use $\hat{\lambda}^{\cD}$, $\hat{\lambda}^{\cN}$ and $\hat{\lambda}^{\cR}$ to denote the principal eigenvalue of the eigenvalue problem \eqref{equ:eig:main} with $\cM$ replaced by $\hat{\cM}$ subject to the Dirichlet, Neumann and Robin boundary conditions, repsectively.
	Then $ \lambda^{\cN} \leq \lambda^{\cR} \leq \lambda^{\cD} \leq \lambda^{\cD_0}$, $\hat{\lambda}^{\cD} \leq \lambda^{\cD} $, $\hat{\lambda}^{\cN} \leq \lambda^{\cN} $ and $\hat{\lambda}^{\cR} \leq \lambda^{\cR} $.
\end{proposition}

\begin{lemma}\label{lem:infty:Neumann:estimate}
	Assume that {\rm (SM)} holds.
	Let $\lambda_{\bmkp}^\cN$ be the principal eigenvalue of the eigenvalue problem \eqref{equ:eig:main} subject to the Neumann boundary condition. Then 
	$$-n \overline{m} \leq \lambda_{\bmkp}^\cN \leq 0,$$
	where $\overline{m}=\max_{1 \leq i,j \leq n} \max_{(x,t) \in \oOmega \times \Real} m_{ij} (x,t)$.
\end{lemma}
\begin{proof}
	Define two matrices $\overline{\cM}=(\overline{m}_{ij})_{n \times n}$ and $\underline{\cM}=(\underline{m}_{ij})_{n \times n}$ with $\overline{m}_{ij}=\overline{m}$ and $\underline{m}_{ij}=0$, $\forall 1 \leq i,j \leq n$. Clearly, $\underline{m}_{ij} \leq m_{ij} \leq \overline{m}_{ij}$, $1 \leq i,j \leq n$. It then follows that $0$ and $- n \overline{m}$ are the principal eigenvalue of the eigenvalue problem \eqref{equ:eig:main} subject to the Neumann boundary condition with $\cM$ replaced by $\underline{\cM}$ and $\overline{\cM}$, respectively. Proposition \ref{prop:eig:com} yields that $-n \overline{m} \leq \lambda_{\bmkp}^\cN \leq 0$.
\end{proof}

Let $\tilde{\cM}=(\tilde{m}_{ij})_{n\times n}$ be a matrix-valued function of $t \in \Real$ with $\tilde{m}_{ij}(t) = \frac{1}{\vert \Omega \vert} \int_{\Omega} m_{ij}(x,t) \D x$, $t \in \Real$, $1 \leq i,j \leq n$, where $\vert \Omega \vert$ is the volume of $\Omega$. It is easy to see that $\tilde{m}_{ij}(t)=\tilde{m}_{ij}(t+T)$, $t \in \Real$, $1 \leq i,j \leq n$. Let $\{\tilde{O}(t,s): t \geq s\}$ be the evolution family on $\Real^n$ of $\frac{\D \bmv}{\D t}=\tilde{\cM}(t) \bmv,~ t>0$ and write $\tilde{\eta}=\omega(\tilde{O})$. 
The subsequent observation  is inspired  by the results in  \cite{hutson2001evolution,peng2015effects}.
\begin{lemma}\label{lem:inf:irreducible}
	Assume that {\rm (SM)} holds.
	Let $\lambda_{\bmkp}^\cN$ be the principal eigenvalue of the eigenvalue problem \eqref{equ:eig:main} subject to the Neumann boundary condition. We further assume that $\tilde{O}(T,0)$ is irreducible. Then we have
	$$
	\lim_{\min_{1 \leq i \leq n} \kappa_i \rightarrow +\infty} \lambda_{\bmkp}^\cN = - \tilde{\eta}.
	$$
\end{lemma}
\begin{proof}
	Let $\bmu_{\bmkp}=(u_{\bmkp,1},u_{\bmkp,2},\cdots,u_{\bmkp,n})^T \in \bX$ be the principal eigenfunction corresponding to $\lambda_{\bmkp}^{\cN}$, where we normalize $\bmu_{\bmkp}$ by
	\begin{equation}\label{equ:normalize}
	\max_{1 \leq i \leq n}\int_{0}^{T}\int_{\Omega} u_{\bmkp,i}^2 \D x \D t =1.
	\end{equation}
	Define $\tilde{\bmu}_{\bmkp}=(\tilde{u}_{\bmkp,1},\tilde{u}_{\bmkp,2},\cdots,\tilde{u}_{\bmkp,n})^T$ and $\hat{\bmu}_{\bmkp}=(\hat{u}_{\bmkp,1},\hat{u}_{\bmkp,2},\cdots,\hat{u}_{\bmkp,n})^T$ by 
	\begin{equation}\label{equ:tilde:hat:u}
	\tilde{u}_{\bmkp,i} (t) = \frac{1}{\vert \Omega \vert} \int_{\Omega} u_{\bmkp,i}(y,t) \D y, ~\hat{u}_{\bmkp,i} (x,t) =u_{\bmkp,i}(x,t) - \tilde{u}_{\bmkp,i} (t), ~ 1 \leq i \leq n.
	\end{equation}
	Taking the average of $i$-th equation of the eigenvalue problem \eqref{equ:eig:main} over $\Omega$, we have
	$$
	\frac{\partial \tilde{u}_{\bmkp,i}}{\partial t}=  \frac{1}{\vert \Omega \vert}\sum_{j=1}^n \int_{\Omega} m_{ij} u_{\bmkp,j} \D x+ \lambda_{\bmkp}^{\cN} \frac{1}{\vert \Omega \vert}\int_{\Omega}  u_{\bmkp,i} \D x,~ 1 \leq i \leq n,
	$$
	and hence,
	$$
	\frac{\partial \tilde{u}_{\bmkp,i}}{\partial t}= \sum_{j=1}^n \tilde{m}_{ij} \tilde{u}_{\bmkp,j}
	+\frac{1}{\vert \Omega \vert} \sum_{j=1}^n \int_{\Omega} m_{ij} (u_{\bmkp,j} - \tilde{u}_{\bmkp,j}) \D x + \lambda_{\bmkp}^{\cN} \tilde{u}_{\bmkp,i}, ~1 \leq i \leq n.
	$$
	Let $\bm{f}_{\bmkp}=(f_{\bmkp,1},f_{\bmkp,2},\cdots,f_{\bmkp,n})^T$, where 
	\begin{equation}\label{equ:f:Di}
	f_{\bmkp,i}(t)=\frac{1}{\vert \Omega \vert} \sum_{j=1}^n \int_{\Omega} m_{ij} (x,t)(u_{\bmkp,j}(x,t) - \tilde{u}_{\bmkp,j}(t)) \D x,~ t \in \Real, ~1 \leq i \leq n.
	\end{equation}
	By the constant-variation formula, we obtain
	\begin{equation}\label{equ:tilde:bmu}
	\tilde{\bmu}_{\bmkp}(t)= e^{t\lambda_{\bmkp}^{\cN}}\tilde{O}(t,0) \tilde{\bmu}_{\bmkp}(0) + \int_{0}^{t} e^{(t-s)\lambda_{\bmkp}^{\cN}} \tilde{O}(t,s) \bm{f}_{\bmkp}(s) \D s, ~t \in [0,T].
	\end{equation}
	Taking $t=T$ and using the fact $\tilde{\bmu}_{\bmkp}(T)=\tilde{\bmu}_{\bmkp}(0)$, we arrive at
	$$
	\tilde{\bmu}_{\bmkp}(0)= e^{T\lambda_{\bmkp}^{\cN}}\tilde{O}(T,0) \tilde{\bmu}_{\bmkp}(0) + \int_{0}^{T} e^{(T-s)\lambda_{\bmkp}^{\cN}} \tilde{O}(T,s) \bm{f}_{\bmkp}(s) \D s.
	$$
	 For convenience, let 
	$
	\underline{\kappa}:=\min_{1 \leq i \leq n} \kappa_i$. Then we have the following claims.
	
	{\it Claim 1.}
	$
	\left\Vert \int_{0}^{t} e^{(t-s)\lambda_{\bmkp}^{\cN}} \tilde{O}(t,s) \bm{f}_{\bmkp}(s) \D s \right\Vert_{\Real^n} \rightarrow 0 
	\text{ as } \underline{\kappa} \rightarrow +\infty,\forall t \in [0,T]. 
	$
	
	{\it Claim 2.}
	$
	\liminf_{\underline{\kappa} \rightarrow +\infty} \sum_{j=1}^n \tilde{u}_{\bmkp,j}(0) >0.
	$
	
	{\it Claim 3.}
	$
	\limsup_{\underline{\kappa} \rightarrow +\infty} \max_{1 \leq j \leq n} \tilde{u}_{\bmkp,j}(0) <+\infty.
	$
	
	Let us postpone the proof of these claims, and complete the proof in a few lines. 
	In view of Lemma \ref{lem:infty:Neumann:estimate}, and Claims 2 and 3, we see that for any $\bmkp_l=(\kappa_{l,1},\kappa_{l,2},\cdots,\kappa_{l,n})^T$, there exists a sequence $\bmkp_{l_k}=(\kappa_{l_k,1},\kappa_{l_k,2},\cdots,\kappa_{l_k,n})^T$ such that $\tilde{\bmu}_{\bmkp_{l_k}}(0) \rightarrow \bm{\phi}$ and $\lambda_{\bmkp_{l_k}}^{\cN} \rightarrow \tilde{\lambda}$ as $\min_{1 \leq i \leq n} \kappa_{l_k,i} \rightarrow +\infty$ for some $\bm{\phi} \in \Real^n_+$ and $\tilde{\lambda}$ with $\bm{\phi} \neq 0$ and $-n \overline{m}\leq \tilde{\lambda}\leq 0$, where $\overline{m}:=\max_{1 \leq i,j \leq n} \max_{(x,t) \in \oOmega \times \Real} m_{ij}(x,t)$.
	It follows from Claim 1 that
	$$
	\bm{\phi}=e^{\tilde{\lambda}T}\tilde{O}(T,0) \bm{\phi}.
	$$
	Since $\tilde{O}(T,0)$ is irreducible, the Perron-Frobenius theorem (see, e.g., \cite[Theorem 4.3.1]{smith2008monotone}) implies that 
	$\tilde{\lambda}= -\frac{1}{T} \ln r (O(T,0))= - \tilde{\eta}$.
	Thus, the desired conclusion follows from the arbitrariness of $\bmkp_l$.
	
	To prove Claim 1, we proceed with three steps.
	
	{\it Step 1.} There is $C_0 >0$ dependent on $\underline{a}$, $\overline{m}$ and $n$ such that
	\begin{equation}\label{equ:nabla:u}
	\int_{0}^{T}\int_{\Omega} \vert \nabla u_{\bmkp,i} \vert^2 \D x \D t
	\leq C_0\kappa_i^{-1}.
	\end{equation}
	Choose $\underline{a} >0$ such that $\sum_{p,q=1}^{N} a_{pq}^{i} \xi_p \xi_q \geq \underline{a} \sum_{p=1}^{N} \xi_p^2$, $ \forall 1 \leq i \leq n$. We multiply the $i$-th equation of the eigenvalue problem \eqref{equ:eig:main} by $u_{\bmkp,i}$ and then integrate over $\Omega \times (0,T)$ to obtain
	$$
	\int_{0}^{T}\int_{\Omega} \kappa_i \sum_{p,q=1}^{N} a_{pq}^{i} (u_{\bmkp,i})_{x_p} (u_{\bmkp,i})_{x_q} \D x \D t
	= \sum_{j=1}^n\int_{0}^{T}\int_{\Omega} m_{ij} u_{\bmkp,j} u_{\bmkp,i} \D x \D t 
	+ \lambda_{\bmkp}^{\cN} \int_{0}^{T}\int_{\Omega} u_{\bmkp,i}^2 \D x \D t.
	$$
	We then have
	$$
	\begin{aligned}
	\underline{a} \kappa_i\int_{0}^{T}\int_{\Omega}  \vert \nabla u_{\bmkp,i} \vert^2 \D x \D t
	\leq&\overline{m} 
	\sum_{j=1}^n \left( \int_{0}^{T}\int_{\Omega} u_{\bmkp,j}^2 \D x \D t\right)^{\frac{1}{2}}
	\left(\int_{0}^{T}\int_{\Omega} u_{\bmkp,i}^2 \D x \D t\right)^{\frac{1}{2}}\\
	&+\lambda_{\bmkp}^{\cN} \int_{0}^{T}\int_{\Omega} u_{\bmkp,i}^2 \D x \D t,~ 1 \leq i \leq n.
	\end{aligned}
	$$
	This  finishes the proof of Step 1.
	
	{\it Step 2.}  For each $i =1,2,\cdots,n$, we have
	\begin{equation}\label{equ:f:bmkp:i}
			\int_{0}^{t} \vert f_{\bmkp,i} (s) \vert \D s  \rightarrow 0 \text{ as } \underline{\kappa} \rightarrow +\infty.
	\end{equation}
	
	It is easy to see that
	$
	\int_{\Omega} \hat{u}_{\bmkp,i} (x,t) \D x =0
	$, $\forall t \in \Real, ~i=1,2,\cdots, n$. By Poincar\'e's inequality (see, e.g., \cite[Theorem 5.8.1]{evans2010partial}), we have
	$$
	\int_{\Omega} \vert \hat{u}_{\bmkp,i} \vert^2 \D x \leq C_1 	\int_{\Omega} \vert \nabla \hat{u}_{\bmkp,i} \vert^2 \D x,~ \forall 1 \leq i \leq n,~ t \in \Real,
	$$
	for some $C_1$ dependent on $\Omega$. Therefore, by using of $\nabla \hat{u}_{\bmkp,i} = \nabla u_{\bmkp,i}$ and \eqref{equ:nabla:u}, we obtain 
	\begin{equation}\label{equ:u:hat:2}
	\int_{0}^{T}\int_{\Omega} \vert \hat{u}_{\bmkp,i} \vert^2 \D x \D t
	\leq C_1 C_0\kappa_i^{-1},~ 1 \leq i \leq n.
	\end{equation}
	Thus, the Caughy inequality yields
	\begin{equation}\label{equ:u:hat}
	\int_{0}^{T}\int_{\Omega} \vert \hat{u}_{\bmkp,i} \vert \D x \D t
	\leq C_2\kappa_i^{-\frac{1}{2}},~ 1 \leq i \leq n,
	\end{equation}
	for some $C_2$ dependent on $\underline{a}$, $\overline{m}$, $n$ and $\Omega$. By \eqref{equ:f:Di}, \eqref{equ:tilde:hat:u} and \eqref{equ:u:hat}, it follows that for any $t \in [0,T]$
	$$
	 \int_{0}^{t} \vert f_{\bmkp,i} (s) \vert \D s
	\leq \overline{m}n\vert \Omega \vert^{-1}  \max_{1 \leq j \leq n}	\int_{0}^{t}\int_{\Omega} \vert \hat{u}_{\bmkp,j} \vert \D x \D t \leq \overline{m}n\vert \Omega \vert^{-1}   C_2\underline{\kappa}^{-\frac{1}{2}},~ 1 \leq i \leq n.
	$$
	This shows that \eqref{equ:f:bmkp:i} holds true.
	
	{\it Step 3.} Complete of the proof of Claim 1.
	For convenience, we equip $\Real^n$ with the maximum norm $\Vert \bmr \Vert_{\Real^n}= \max_{1\leq i \leq n} \vert r_i \vert$, where $\bmr =(r_1,r_2,\cdots,r_n)^T$. For each $i=1,2,\cdots,n$, let $[e^{(t-s)\lambda_{\bmkp}^{\cN}} \tilde{O}(t,s) \bm{f}_{\bmkp}(s) ]_{i}$ be the $i$-component of $e^{(t-s)\lambda_{\bmkp}^{\cN}} \tilde{O}(t,s) \bm{f}_{\bmkp}(s) $, it then easily follows that the $i$-component of $\int_{0}^{t}e^{(t-s)\lambda_{\bmkp}^{\cN}} \tilde{O}(t,s) \bm{f}_{\bmkp}(s) \D s$ is
	$$
	\left[ \int_{0}^{t} 
	e^{(t-s)\lambda_{\bmkp}^{\cN}} \tilde{O}(t,s) \bm{f}_{\bmkp}(s) \D s \right]_{i}=  \int_{0}^{t} 
	\left[e^{(t-s)\lambda_{\bmkp}^{\cN}} \tilde{O}(t,s) \bm{f}_{\bmkp}(s) \right]_{i} \D s.
	$$
	 By \cite[Lemma 5.2]{daners1992abstract}, there exists $C_3>0$ such that 
	$$
	\left\Vert e^{(t-s)\lambda_{\bmkp}^{\cN}}  \tilde{O}(t,s) \bm{f}_{\bmkp}(s) \right\Vert_{\Real^n} \leq C_3 \left\Vert  \bm{f}_{\bmkp}(s) \right\Vert_{\Real^n},
	$$
	that is,
	$$
	\max_{1 \leq i \leq n} \left\vert [ e^{(t-s)\lambda_{\bmkp}^{\cN}}  \tilde{O}(t,s) \bm{f}_{\bmkp}(s) ]_i \right\vert
	\leq C_3 \max_{1 \leq i \leq n} \vert f_{\bmkp,i} (s) \vert,
	$$
	 uniformly for $s,t \in [0,T]$ with $s \leq t$.
	This implies that 
	$$
	\max_{1 \leq i \leq n}
	\left\vert
	\left[ \int_{0}^{t} 
	e^{(t-s)\lambda_{\bmkp}^{\cN}} \tilde{O}(t,s) \bm{f}_{\bmkp}(s) \D s 
	\right]_i
	\right\vert
	\leq C_3 \max_{1 \leq i \leq n} \int_{0}^{t} \vert f_{\bmkp,i} (s) \vert \D s.
	$$
	We complete the proof of Claim 1 due to \eqref{equ:f:bmkp:i}.
	
	Now we turn to the proof of Claim 2. Suppose that it is false, then 	$$\liminf_{\underline{\kappa}\rightarrow +\infty} \sum_{j=1}^n \tilde{u}_{\bmkp,j}(0) =0.$$
	By Claim 1 and \eqref{equ:tilde:bmu}, it follows that
	$$\liminf_{\underline{\kappa}\rightarrow +\infty} \sum_{j=1}^n \tilde{u}_{\bmkp,j}(t) =0, \text{ uniformly for } t \in [0,T],$$
	which implies that 
	$$
	\liminf_{\underline{\kappa}\rightarrow +\infty}
	\sum_{j=1}^n
	\int_{0}^{T} \int_{\Omega} \tilde{u}_{\bmkp,j}^2(t) \D x \D t
	=0. 
	$$
	Note that for each $j=1,2,\cdots, n$,
	$$
	\int_{0}^{T} \int_{\Omega} u_{\bmkp,j}^2(t) \D x \D t
	\leq
	2\left[\int_{0}^{T} \int_{\Omega} \hat{u}_{\bmkp,j}^2(t) \D x \D t
	+
	\int_{0}^{T} \int_{\Omega} \tilde{u}_{\bmkp,j}^2(t) \D x \D t \right].
	$$
	Together with \eqref{equ:u:hat:2}, we then have
	$$
	\liminf_{\underline{\kappa}\rightarrow +\infty}\max_{1 \leq j \leq n} \int_{0}^{T} \int_{\Omega} u_{\bmkp,j}^2(t) \D x \D t=0,
	$$
	which contrary to \eqref{equ:normalize}.
	
	Finally, we prove Claim 3. Assume, by contradiction, that
	$$
	\limsup_{\underline{\kappa}\rightarrow +\infty} \max_{1 \leq j \leq n} \tilde{u}_{\bmkp,j}(0) =+\infty.
	$$
	Then Claim 1 and \eqref{equ:tilde:bmu} yield
	$$
	\limsup_{\underline{\kappa}\rightarrow +\infty} \max_{1 \leq j \leq n} \tilde{u}_{\bmkp,j}(t) =+\infty, \text{ uniformly for } t \in [0,T].
	$$
	Since
	$$
	\int_{0}^{T} \tilde{u}_{\bmkp,j}^2 \D t
	=
	\int_{0}^{T} \left( \frac{1}{\vert \Omega \vert}\int_{\Omega} u_{\bmkp,j} \D x \right)^2 \D t
	\leq \frac{1}{\vert \Omega \vert}
	\int_{0}^{T} \int_{\Omega} u_{\bmkp,j}^2 \D x \D t,
	$$
	we obtain
	$$
	\limsup_{\underline{\kappa}\rightarrow +\infty} \max_{1 \leq j \leq n} \int_{0}^{T} \int_{\Omega} u_{\bmkp,j}^2 \D x \D t=+\infty,
	$$
	which contradicts \eqref{equ:normalize}.
\end{proof}

To remove the irreducibility condition on $\tilde{O}(T,0)$  in Lemma \ref{lem:inf:irreducible}, we analyze the block of $\tilde{O}(T,0)$ when it is reducible. 
\begin{lemma}\label{lem:a_i0j0}
	Assume that {\rm (SM)} holds.
	Let $A:=\tilde{O}(T,0)=(a_{ij})_{n\times n}$. If $a_{i_0j_0} =0$ for some $1 \leq i_0 \neq j_0 \leq n$, then $\tilde{m}_{i_0j_0}(t)=0$, $\forall t \in \Real$. Furthermore, $m_{i_0j_0}(x,t)=0$, $\forall (x,t) \in \oOmega \times \Real$.
\end{lemma}

\begin{proof}
 Suppose, by contradiction, that there exists $t_0 \in [0,T]$ such that $\tilde{m}_{i_0j_0} (t_0) >0$. Let $\bme^{i}=(0,\cdots,1,\cdots,0)^T$ be a unit $n$-dimensional vector whose only $i$-th component is nonzero. For each $i=1,2,\cdots, n$, we use $\bmv^{i}(t)=(v_{1}^{i},v_{2}^{i},\cdots,v_{n}^{i})^T (t)$ to denote the solution of $\frac{\D \bmu}{\D t}=\tilde{\cM}(t) \bmu$ with initial data $\bmv^{i}(0)=\bme^{i}$. Clearly, $v_{j_0}^{j_0}(t)>0$ for all $t \in [0,T]$, which implies 
$$
\frac{\D }{\D t} v_{i_0}^{j_0}(t_0) = \sum_{j =1}^{n} \tilde{m}_{i_0j}(t_0) v_j^{j_0}(t_0)>0.
$$
Thus, we obtain $a_{i_0 j_0}=v_{i_0}^{j_0}(T)>0$, which contradicts our assumption.
\end{proof}
\begin{lemma}\label{lem:split}
	Assume that {\rm (SM)} holds.
	Write $A:=\tilde{O}(T,0)=(a_{ij})_{n\times n}$ and let
	$$
	A=
	\left(
	\begin{matrix}
	A_{11} & A_{12} & \cdots & A_{1 \tn}\\
	A_{21} & A_{22}& \cdots & A_{2 \tn}\\
	\vdots & \vdots&\ddots &\vdots\\
	A_{\tn 1} & A_{\tn 2} & \cdots & A_{\tn \tn}
	\end{matrix}
	\right),
	\text{ and }
	\tilde{\cM}=
	\left(
	\begin{matrix}
	\tilde{\cM}_{11} & \tilde{\cM}_{12} & \cdots & \tilde{\cM}_{1 \tn}\\
	\tilde{\cM}_{21} & \tilde{\cM}_{22}& \cdots & \tilde{\cM}_{2 \tn}\\
	\vdots & \vdots&\ddots &\vdots\\
	\tilde{\cM}_{\tn 1} & \tilde{\cM}_{\tn 2} & \cdots & \tilde{\cM}_{\tn \tn}
	\end{matrix}
	\right),
	$$
	where $A_{kk}$ is an $i_k \times i_k$ matrix, and $\tilde{\cM}_{kk}$ is an $i_k \times i_k$ matrix-valued function of $t \in \Real$. If $A_{kl}$ are zero matrices for all $ k =1,2,\cdots, \tn$ and $l > k$, then so are $\tilde{\cM}_{kl}(t)$ for any $ t\in \Real$. Moreover, let $\{\tilde{O}_k(t,s): t \geq s\}$ be the evolution family of $\frac{\D \bmv}{\D t}=\tilde{\cM}_{kk}(t) \bmv,~t>0$, then $\omega(\tilde{O}_{k})= \frac{\ln r(A_{kk})}{T}$.
\end{lemma}

\begin{proof}
	Without loss of generality, we assume that $\tn=2$ and
	$$
	A=
	\left(
	\begin{matrix}
	A_{11} &A_{12} \\
	A_{21} & A_{22}\\
	\end{matrix}
	\right),
	\text{ and }
	\tilde{\cM}=
	\left(
	\begin{matrix}
	\tilde{\cM}_{11} & \tilde{\cM}_{12} \\
	\tilde{\cM}_{21} & \tilde{\cM}_{22}\\
	\end{matrix}
	\right),
	$$
	Then Lemma \ref{lem:a_i0j0} implies that $\tilde{\cM}_{12}$ is a zero matrix.
	We next claim that
	$\tilde{O}_1 (T,0)=A_{11}$ and $\tilde{O}_2 (T,0)=A_{22}$. If the claim is true, then Lemma \ref{lem:w_theta:equ} yields that $\omega(\tilde{O}_{k})= - \frac{\ln r(A_{kk})}{T}$, $1 \leq k \leq 2$.
	
	It suffices to prove the above claim. Let $\bme^{i}=(0,\cdots,1,\cdots,0)^T$, $\bme^{i,1}=(0,\cdots,1,\cdots,0)^T$ and $\bme^{i,2}=(0,\cdots,1,\cdots,0)^T$ be three unit $n$, $i_1$ and $i_2$ dimensional vectors whose only $i$-th component is nonzero, respectively. 
	Let $\Lambda_1: =\{1,2,\cdots,i_1\}$ and $\Lambda_2:=\{i_1+1,i_1+2,\cdots,n\}$.
	We use 
	$$\bmv^{i}(t)=(v_{1}^{i},v_{2}^{i},\cdots,v_{n}^{i})^T (t), ~i \in n,$$
	$$\bmv^{i,1}(t)=(v_{1}^{i,1},v_{2}^{i,1},\cdots,v_{i_1}^{i,1})^T(t), ~i \in \Lambda_1,$$
	and 
	$$\bmv^{i-i_1,2}(t)=(v_{1}^{i-i_1,2},v_{2}^{i-i_1,2},\cdots,v_{i_2}^{i-i_1,2})^T(t), ~i \in \Lambda_2,$$
	to denote the solution of $\frac{\D \bmu}{\D t}=\tilde{\cM}(t) \bmu$, $\frac{\D \bmu}{\D t}=\tilde{\cM}_{11}(t) \bmu$ and $\frac{\D \bmu}{\D t}=\tilde{\cM}_{22}(t) \bmu$ with initial data $\bmv^{i}(0)=\bme^{i}$, $\bmv^{i,1}(0)=\bme^{i,1}$ and $\bmv^{i-i_1,2}(0)=\bme^{i-i_1,2}$, respectively. 
	It is easy to verify that 	$$\tilde{O} (T,0)= (\bmv^{1}(T),\bmv^{2}(T),\cdots,\bmv^{n}(T)),$$ 
	$$\tilde{O}_1 (T,0)= (\bmv^{1,1}(T),\bmv^{2,1}(T),\cdots,\bmv^{i_1,1}(T)),$$
	and
	$$\tilde{O}_2 (T,0)= (\bmv^{1,2}(T),\bmv^{2,2}(T),\cdots,\bmv^{i_2,2}(T)).$$ 
	By the uniqueness of solution, we obtain $v_{j}^{i,1}(T)=v_{j}^{i}(T)=a_{ji}$, $i,j \in \Lambda_1$, that is, $\tilde{O}_1 (T,0)=A_{11}$. Since $A_{12}$ is a zero matrix, we have that $v_{j}^{i}(T)=a_{ji}=0$, $i \in \Lambda_2$, $j \in \Lambda_1$, and hence, $v_{j}^{i}(t)=0$, $t\in [0,T]$, $i \in \Lambda_2$, $j \in \Lambda_1$. Therefore, the uniqueness of solution implies  that $v_{j-i_1}^{i-i_1,2}(T)=v_{j}^{i}(T)=a_{ji}$, $i,j \in \Lambda_2$, that is, $\tilde{O}_2 (T,0)=A_{22}$.
\end{proof}

The following two lemmas can be derived from Lemma \ref{lem:lower_upper} and the intermediate value theorem, respectively.
\begin{lemma}\label{lem:eig:part}
	Assume that {\rm (SM)} holds.
	For any given $\Lambda \subset \{1,2,\cdots,n\}$, let $\lambda_{\Lambda}^{\cN}$ be the principal eigenvalue of the periodic parabolic eigenvalue problem
	$$
	\begin{cases}
	\frac{\partial u_i}{\partial t}=  \cL_i (x,t)u_i + \sum_{j\in \Lambda} m_{ij}(x,t) u_j +\lambda u_i, &(x,t) \in \Omega \times \Real, ~i \in \Lambda,\\
	\cB_i u_i=0, &(x,t) \in \partial\Omega \times \Real,~ i \in \Lambda,
	\end{cases}
	$$
	where  $\cB_i$ satisfies the Neumann boundary condition.
	Let $\lambda^{\cN}$ be the principal eigenvalue of the eigenvalue problem \eqref{equ:eig:main} subject to the Neumann boundary condition. Then $\lambda^{\cN} \leq \lambda_{\Lambda}^{\cN}$.
\end{lemma}

\begin{lemma}\label{lem:conti_limits}
	Let $g$ be a continuous function on $(a,b)$ and write $g_+=\limsup_{x \rightarrow b} g(x)$ and $g_-=\liminf_{x \rightarrow b} g(x)$. Then for any $c \in [g_-,g_+]$, there exists a sequence $x_k \in (a,b)$ such that $\lim\limits_{x_k \rightarrow b} g(x_k)=c$.
\end{lemma}

Now we are in a position to prove the main results of this section.

\begin{theorem}\label{thm:infty:Neumann}
	Assume that {\rm (M)} holds.
	Let $\lambda_{\bmkp}^\cN$ be the principal eigenvalue of the eigenvalue problem \eqref{equ:eig:main} subject to the Neumann boundary condition. Then we have
	$$
	\lim_{\min_{1 \leq i \leq n} \kappa_i \rightarrow +\infty} \lambda_{\bmkp}^\cN = - \tilde{\eta}.
	$$
\end{theorem}

\begin{proof}
	Without lose of generality, we assume that (SM) holds due to Lemma \ref{lem:transfer}.
	The conclusion has been proved in the case where $\tilde{O}(T,0)$ is irreducible in Lemma \ref{lem:inf:irreducible}, so we assume that $\tilde{O}(T,0)$ is reducible.
	
	We first show that $\lambda_{\infty}:=\lim_{\min_{1 \leq i \leq n} \kappa_i \rightarrow +\infty} \lambda_{\bmkp}^\cN $ exists. According to Lemma \ref{lem:infty:Neumann:estimate},
	$\lambda_+:=\limsup_{\min_{1 \leq i \leq n} \kappa_i \rightarrow +\infty} \lambda_{\bmkp}^\cN$ and $\lambda_-:=\liminf_{\min_{1 \leq i \leq n} \kappa_i \rightarrow +\infty} \lambda_{\bmkp}^\cN$ exist, and $-n \overline{m}\leq \lambda_-,\lambda_+ \leq 0$. It suffices to prove that $\lambda_-=\lambda_+$. Suppose that
	$\lambda_-< \lambda_+$, for any $\tilde{\lambda} \in [\lambda_-, \lambda_+]$,
	by Lemma \ref{lem:conti_limits} and Claims 2 and 3 in the proof of Lemma \ref{lem:inf:irreducible}, there exists a sequence $\bmkp_{l}=(\kappa_{l,1},\kappa_{l,2},\cdots,\kappa_{l,n})^T$ such that $\tilde{\bmu}_{\bmkp_{l}}(0) \rightarrow \bm{\phi}$ and $\lambda_{\bmkp_{l}}^{\cN} \rightarrow \tilde{\lambda}$ as $\min_{1 \leq i \leq n} \kappa_{l,i} \rightarrow +\infty$ for some $\bm{\phi} \in \Real^n_+$ with $\bm{\phi} \neq 0$.
	By repeating the argument in the proof of Lemma \ref{lem:inf:irreducible}, we have
	$$
	\bm{\phi}=e^{\tilde{\lambda}T}\tilde{O}(T,0) \bm{\phi}.
	$$
	This implies that $e^{-\tilde{\lambda}T}$ is an eigenvalue of $\tilde{O}(T,0)$ for any $\tilde{\lambda} \in [\lambda_-, \lambda_+]$, which is impossible.
	
	Now it suffices to show that $\lambda_{\infty}= - \tilde{\eta}$. We first prove that $\lambda_{\infty}\geq - \tilde{\eta}$. For any given $\epsilon>0$, let $\cM^{\epsilon}=(m_{ij}^{\epsilon})_{n \times n}$ and $\tilde{\cM}^{\epsilon}=(\tilde{m}_{ij}^{\epsilon})_{n \times n}$ be two continuous matrix-valued functions of $(x,t) \in \oOmega \times \Real$ and $t \in \Real$ with $m_{ij}^{\epsilon}(x,t)=m_{ij}(x,t)+\epsilon$, $\forall (x,t) \in \oOmega \times \Real$ and $\tilde{m}_{ij}^{\epsilon}(x,t)=\frac{1}{\vert \Omega \vert} \int_{\Omega} m_{ij}^{\epsilon} (x,t) \D x$, $\forall t \in \Real$. Let $\{\tilde{O}^{\epsilon}(t,s): t \geq s\}$ be the evolution family on $\Real^n$ of $\frac{\D \bmv}{\D t}=\tilde{\cM}^{\epsilon}(t) \bmv, ~t>0$ and $\tilde{\eta}^{\epsilon}:=\omega(\tilde{O}^{\epsilon})$. It is easy to see that $\tilde{\eta}^{\epsilon} \rightarrow \tilde{\eta}$ as $\epsilon \rightarrow 0$. Let $\lambda_{\bmkp,\epsilon}^\cN$ be the principal eigenvalue of the eigenvalue problem \eqref{equ:eig:main} subject to the Neumann boundary condition with $\cM$ replaced by $\cM^{\epsilon}$. Clearly, $\lambda_{\bmkp,\epsilon}^{\cN} \leq \lambda_{\bmkp}^{\cN}$, $\forall \epsilon>0$. Since
	$$-\tilde{\eta}^{\epsilon}=\lim\limits_{\min_{1 \leq i \leq n} \kappa_i \rightarrow +\infty} \lambda_{\bmkp,\epsilon}^{\cN} \leq \lim\limits_{\min_{1 \leq i \leq n} \kappa_i \rightarrow +\infty} \lambda_{\bmkp}^{\cN}=\lambda_{\infty}, \forall \epsilon >0,$$
	it follows that $-\tilde{\eta} \leq \lambda_{\infty}$.
	It remains to prove that $\lambda_{\infty} \leq - \tilde{\eta}$. By Lemma \ref{lem:reducible}, without loss of generality, we assume that 
	$$
	\tilde{O}(T,0)=
	\left(
	\begin{matrix}
	A_{11} & 0 & \cdots & 0\\
	A_{21} & A_{22}& \cdots & 0\\
	\vdots & \vdots&\ddots &\vdots\\
	A_{\tn 1} & A_{\tn 2} & \cdots & A_{\tn \tn}
	\end{matrix}
	\right),
	$$
	where $A_{kk}$ is an $i_k \times i_k$ irreducible matrix and $\sum_{k=1}^{\tn} i_k =n$. Moreover, $r(\tilde{O}(T,0))= \max_{1 \leq k \leq \tn} r(A_{kk})$. Let $\Lambda_1: = \{i \in \bZ: 1 \leq i \leq i_1\}$ and $\Lambda_k:= \{i \in \bZ: \sum_{l=1}^{k-1} i_l+1 \leq i \leq \sum_{l=1}^{k} i_l \}$, $2 \leq k \leq \tn$. We split $\cM$ and $\tilde{\cM}$ into
	$$
	\cM=
	\left(
	\begin{matrix}
	\cM_{11} & \cM_{12} & \cdots & \cM_{1 \tn}\\
	\cM_{21} & \cM_{22}& \cdots & \cM_{2 \tn}\\
	\vdots & \vdots&\ddots &\vdots\\
	\cM_{\tn 1} & \cM_{\tn 2} & \cdots & \cM_{\tn \tn}
	\end{matrix}
	\right)
	\text{ and }
	\tilde{\cM}=
	\left(
	\begin{matrix}
	\tilde{\cM}_{11} & \tilde{\cM}_{12} & \cdots & \tilde{\cM}_{1 \tn}\\
	\tilde{\cM}_{21} & \tilde{\cM}_{22}& \cdots & \tilde{\cM}_{2 \tn}\\
	\vdots & \vdots&\ddots &\vdots\\
	\tilde{\cM}_{\tn 1} & \tilde{\cM}_{\tn 2} & \cdots & \tilde{\cM}_{\tn \tn}
	\end{matrix}
	\right),
	$$
	where $\cM_{kk}$ and $\tilde{\cM}_{kk}$ are $i_k \times i_k$ continuous matrix-valued functions of $(x,t) \in \oOmega \times \Real$ and $t \in \Real$, respectively. For each $k=1,2,\cdots, \tn$, let $\{\tilde{O}_k(t,s): t \geq s\}$ be the evolution family of $\frac{\D \bmv}{\D t}=\tilde{\cM}_{kk}(t) \bmv,~t>0$ and write $\tilde{\eta}_k=\omega(\tilde{O}_{k})$. 
	Then Lemma \ref{lem:split} implies that for each $ k =1,2,\cdots, \tn$ and $l > k$, $ \tilde{\cM}_{kl}$, and hence, $ \cM_{kl}$ is a zero matrix and $\tilde{\eta}_k= \frac{\ln r(A_{kk})}{T}$. So we have $\tilde{\eta}= \frac{\ln r(\tilde{O} (T,0))}{T}=\max_{1 \leq k \leq \tn} \tilde{\eta}_k$. 
	
	For each $k=1,2,\cdots, \tn$, let $\lambda_{\bmkp,k}^{\cN}$ be the principal eigenvalue of the periodic parabolic eigenvalue problem
	$$
	\begin{cases}
	\frac{\partial u_i}{\partial t}= \kappa_i \cL_i (x,t)u_i + \sum_{j\in \Lambda_k} m_{ij}(x,t) u_j +\lambda u_i, &(x,t) \in \Omega \times \Real,~i \in \Lambda_k,\\
	\cB_i u_i=0, &(x,t) \in \partial\Omega \times \Real,~ i \in \Lambda_k.
	\end{cases}
	$$
	Here $\cB_i$ is the Neumann boundary condition given in \eqref{equ:Neumann}.
	In view of Lemma \ref{lem:eig:part}, we obtain $\lambda_{\bmkp,k}^{\cN} \geq \lambda_{\bmkp}^{\cN}$, $1 \leq k \leq \tn$.
	By Lemma \ref{lem:inf:irreducible}, it then follows that
	$$
	\lim_{\min_{1 \leq i \leq n} \kappa_i \rightarrow +\infty} \lambda_{\bmkp,k}^\cN = - \tilde{\eta}_k, ~1 \leq k \leq \tn.
	$$
	Since 
	$$
	- \tilde{\eta}_k=\lim_{\min_{1 \leq i \leq n} \kappa_i \rightarrow +\infty} \lambda_{\bmkp,k}^\cN \geq \lim_{\min_{1 \leq i \leq n} \kappa_i \rightarrow +\infty} \lambda_{\bmkp}^\cN =\lambda_{\infty}, ~1 \leq k \leq \tn,
	$$
	we conclude that $- \tilde{\eta}=\min_{1 \leq k \leq \tn} - \tilde{\eta}_k \geq \lambda_{\infty}$.
\end{proof}

If, in addition, $\cM$ can be written as a block lower triangular matrix-valued function, we have the following observations for the eigenvalue problem \eqref{equ:eig:main}.
\begin{proposition}\label{prop:eig:split}
	Assume that {\rm (SM)} holds, and $\cB$ satisfies the Dirichlet or Neumann or Robin boundary condition.
 If $\cM$ can be split into 
 $$
 \cM=
 \left(
 \begin{matrix}
 \cM_{11} & \cM_{12} & \cdots & \cM_{1 \tn}\\
 \cM_{21} & \cM_{22}& \cdots & \cM_{2 \tn}\\
 \vdots & \vdots&\ddots &\vdots\\
 \cM_{\tn 1} & \cM_{\tn 2} & \cdots & \cM_{\tn \tn}
 \end{matrix}
 \right),
 $$
 where $\cM_{kk}$ is an $i_k \times i_k$ matrix-valued function of $(x,t) \in \oOmega \times \Real$ and $ \cM_{kl} (x,t)$ is a zero matrix for all $ k =1,2,\cdots, \tn$ and $l > k$, $(x,t) \in \oOmega \times \Real$. Let $\Lambda_1:= \{i \in \bZ: 1 \leq i \leq i_1\}$ and $\Lambda_k:= \{i \in \bZ: \sum_{l=1}^{k-1} i_l+1 \leq i \leq \sum_{l=1}^{k} i_l \}$, $2 \leq k \leq \tn$. Then any eigenvalue of the eigenvalue problem \eqref{equ:eig:main} is one of the eigenvalues of the eigenvalue problem
 \begin{equation}\label{equ:eig:Gamma:k}
 	\begin{cases}
 	\frac{\partial u_i}{\partial t}= \kappa_i \cL_i (x,t)u_i + \sum_{j\in \Lambda_k} m_{ij}(x,t) u_j +\lambda u_i, &(x,t) \in \Omega \times \Real,~i \in \Lambda_k,\\
 	\cB_i u_i =0, &(x,t) \in \partial\Omega \times \Real,~ i \in \Lambda_k.
 	\end{cases}
 \end{equation}
\end{proposition}

\begin{proof}
	Without loss of generality, we assume that $\tn=2$. It then follows that
	$$
	\cM=
	\left(
	\begin{matrix}
	\cM_{11} & 0 \\
	\cM_{21} & \cM_{22}
	\end{matrix}
	\right),
	$$
	where $\cM_{11}$ and $\cM_{22}$ are $i_1 \times i_1$ and $i_2 \times i_2$ matrix-valued continuous functions of $(x,t) \in \oOmega \times \Real$, respectively. Let $\mu$ and $\bmu=(u_1,u_2,\cdots,u_n)^T \in \bX$ be the eigenvalue and eigenfunction of the eigenvalue problem \eqref{equ:eig:main} as the conclusions can be derived by induction method in general. Write $\bmu^{1}=(u_1,u_2,\cdots,u_{i_1})^{T}$ and $\bmu^{2}=(u_{i_1+1},u_{i_1+2},\cdots,u_{n})^{T}$. If $\bmu^{1} \neq 0$, then $\mu$ and $\bmu^{1}$ is the eigenvalue and eigenfunction of \eqref{equ:eig:Gamma:k} in the case where $k=1$. If $\bmu^{1} = 0$, then $\mu$ and $\bmu^{2}$ is the eigenvalue and eigenfunction of \eqref{equ:eig:Gamma:k} in the case where $k=2$. 
\end{proof}

It is worthy to point out that Theorem \ref{thm:infty:Neumann} can also be proved by using Lemmas \ref{lem:a_i0j0} and \ref{lem:split}, Proposition \ref{prop:eig:split} and a slightly modified proof of the part of $\lambda_{\infty} \leq - \tilde{\eta}$. 

Next, we present our second main result of this section.

\begin{theorem}\label{thm:infty:D_R}
	Assume that {\rm (M)} holds.
	Let $\lambda_{\bmkp}^{\cD}$ and $\lambda_{\bmkp}^{\cR}$ be the principal eigenvalue of the eigenvalue problem \eqref{equ:eig:main} subject to the Dirichlet or Robin boundary condition, respectively. Then we have 
	$$
	\lim_{\min_{1 \leq i \leq n} \kappa_i \rightarrow +\infty}\lambda_{\bmkp}^{\cD} =\lim_{\min_{1 \leq i \leq n} \kappa_i \rightarrow +\infty} \lambda_{\bmkp}^{\cR} = +\infty.
	$$
\end{theorem}

\begin{proof}
	Without loss of generality, we assume that (SM) holds due to Lemma \ref{lem:transfer}. By Proposition \ref{prop:eig:com}, we have $\lambda_{\bmkp}^{\cR} \leq \lambda_{\bmkp}^{\cD}$.	It then suffices to show
	$$
	\lim_{\min_{1 \leq i \leq n} \kappa_i \rightarrow +\infty} \lambda_{\bmkp}^{\cR} = +\infty.
	$$
	Choose $\underline{a} >0$ and $\underline{b}>0$ such that $\sum_{p,q=1}^{N} a_{pq}^{i} \xi_p \xi_q \geq \underline{a} \sum_{p=1}^{N} \xi_p^2,~ b_i \geq \underline{b}, ~ \forall 1 \leq i \leq n$.
	Let $\mu_1>0$ be the principal eigenvalue of the eigenvalue problem
	$$
	\begin{cases}
	-\underline{a} \Delta v = \lambda v, & x \in \Omega,\\
	\underline{a} \frac{\partial v}{\partial \bm{\nu}} + \underline{b} v=0, & x \in \partial \Omega.
	\end{cases}
	$$
It  then follows that (see, e.g., \cite[Theorem  2.1]{cantrell2004spatial}) 
	$$
	\mu_1= \inf_{v \in H^1(\Omega)} \frac{\underline{a} \int_{\Omega} \vert \nabla v\vert^2 \D x + \underline{b} \int_{\partial \Omega} v^2 \D x}{\int_{\Omega} v^2 \D x}, 
	$$
and hence, the following inequality holds true:
	\begin{equation}\label{equ:mu_1}
		\underline{a} \int_{\Omega} \vert\nabla v\vert^2 \D x + \underline{b} \int_{\partial \Omega} v^2 \D x \geq \mu_1 	\int_{\Omega} v^2 \D x,~ \forall v \in H^1(\Omega).
	\end{equation}
	For convenience,  we define
	$$
	\underline{\kappa}:=\min_{1 \leq i \leq n} \kappa_i \quad \text{and}\quad \overline{m}:=\max_{1 \leq i,j \leq n} \max_{(x,t) \in \oOmega \times \Real} m_{ij} (x,t).
	$$ 
	Let $\bmu_{\bmkp}=(u_{\bmkp,1},u_{\bmkp,2},\cdots,u_{\bmkp,n})^T \in \bX$ be the principal eigenfunction corresponding to $\lambda_{\bmkp}^{\cN}$.
	We multiply the $i$-th equation of \eqref{equ:eig:main} by $u_{\bmkp,i}$ and integrate over $\Omega \times [0,T]$ to obtain
	$$
		\begin{aligned}
		&  \sum_{j=1}^n \int_{0}^{T} \int_{\Omega} m_{ij}  u_{\bmkp,j} u_{\bmkp,i} \D x \D t
		+ \lambda_{\bmkp}^{\cR} \int_{0}^{T} \int_{\Omega} u_{\bmkp,i}^2 \D x \D t \\
		=&
		 \kappa_i \sum_{p,q=1}^{N} \int_{0}^{T} \int_{\Omega}  a_{pq}^{i} (u_{\bmkp,i})_{x_p} (u_{\bmkp,i})_{x_q} \D x \D t
		+ \kappa_i \int_{0}^{T} \int_{\partial\Omega} b_i u_{\bmkp,i}^2 \D x \D t\\
		\geq &  \underline{a} \kappa_i \int_{0}^{T} \int_{\Omega} \vert\nabla u_{\bmkp,i}\vert^2 \D x \D t
		+ \underline{b} \kappa_i \int_{0}^{T} \int_{\partial\Omega}  u_{\bmkp,i}^2 \D x \D t\\
		\geq &		\mu_1 \underline{\kappa} \int_{0}^{T} \int_{\Omega} u_{\bmkp,i}^2 \D x \D t,
		\end{aligned}
	$$
 	where the last inequality follows from \eqref{equ:mu_1}. On the ther hand, for any $1 \leq i \leq n$,
 	$$
 	\begin{aligned}
 	&\sum_{j=1}^n \int_{0}^{T}\int_{\Omega}  m_{ij} u_{\bmkp,j} u_{\bmkp,i} \D x \D t 
 	+ \lambda_{\bmkp}^{\cR} \int_{0}^{T} \int_{\Omega} u_{\bmkp,i}^2 \D x \D t\\
 	\leq & \overline{m} \sum_{j=1}^n  \left(\int_{0}^{T}\int_{\Omega}  u_{\bmkp,j}^2 \D x \D t\right)^{\frac{1}{2}}
 	\left(\int_{0}^{T}\int_{\Omega}  u_{\bmkp,i}^2 \D x \D t\right)^{\frac{1}{2}}
 	+ \lambda_{\bmkp}^{\cR} \int_{0}^{T} \int_{\Omega} u_{\bmkp,i}^2 \D x \D t.
 	\end{aligned}
 	$$
 	Thus, the above two inequalities give rise to 
 	\begin{equation}\label{equ:mu:kappa}
 		(\mu_1 \underline{\kappa} -  \lambda_{\bmkp}^{\cR}) \left(
 		\int_{0}^{T}   \int_{\Omega} u_{\bmkp,i}^2 \D x \right)^{\frac{1}{2}} 
 		\leq \overline{m}
 		\sum_{j=1}^n  \left(\int_{0}^{T}\int_{\Omega}  u_{\bmkp,j}^2 \D x \D t\right)^{\frac{1}{2}},~ \forall 1 \leq i \leq n.
 	\end{equation}
 	Adding all inequalities of \eqref{equ:mu:kappa} together, we have
 	$$
 	(\mu_1 \underline{\kappa} -  \lambda_{\bmkp}^{\cR}) 	\sum_{i=1}^n  \left(
 	\int_{0}^{T}   \int_{\Omega} u_{\bmkp,i}^2 \D x \right)^{\frac{1}{2}} \leq n\overline{m}
 	\sum_{i=1}^n  \left(\int_{0}^{T}\int_{\Omega}  u_{\bmkp,i}^2 \D x \D t\right)^{\frac{1}{2}},
 	$$
 	and hence, $ \lambda_{\bmkp}^{\cR} \geq \mu_1 \underline{\kappa}-n \overline{m}$, which implies the desired limiting property.
\end{proof}

\section{The basic reproduction ratio}\label{sec:R0}
In this section, we study the asymptotic behavior of the basic reproduction ratio as the diffusion coefficients go to zero and infinity, respectively. 

We use the same notations $\Omega$, $n$, $X$, $X_+$, $\bX$, $\bX_+$, $\bmkp$, $\cL$ and $\cB$ as in the last section. Let $(\cX, \rho)$ be a metric space with metric $\rho$.
For any given $\chi \in \cX$, let $\cV_{\chi}=(v_{ij,\chi})_{n\times n}$ and $\cF_{\chi}=(f_{ij,\chi})_{n\times n}$ be two families of H\"{o}lder continuous $n\times n$ matrix-valued functions of $(x,t) \in \overline{\Omega}\times \Real$ with $\cV_{\chi}(x,t)=\cV_{\chi}(x,t+T)$ and $\cF_{\chi}(x,t)=\cF_{\chi}(x,t+T)$. Let  $\tilde{\cV}_{\chi}=(\tilde{v}_{ij,\chi})_{n\times n}$ and $\tilde{\cF}_{\chi}=(\tilde{f}_{ij,\chi})_{n\times n}$ with $\tilde{v}_{ij,\chi}(t)= \frac{1}{\vert \Omega \vert} \int_{\Omega} v_{ij,\chi} (x,t) \D x$ and $\tilde{f}_{ij,\chi}(t)= \frac{1}{\vert \Omega \vert} \int_{\Omega} f_{ij,\chi} (x,t) \D x$, $\forall t \in \Real$.
We use $\{\Phi_{\chi} (t,s): t \geq s\}$, $\{\tilde{\Phi}_{\chi} (t,s): t \geq s\}$ and $\{\Phi_{\bmkp,\chi}(t,s): t \geq s \}$ to denote the evolution families on $X$ of 
\begin{equation}\label{equ:Phi_0}
	\frac{\partial \bmv}{\partial t}= - \cV_{\chi}(x,t) \bmv,~ x \in \oOmega,~t>0,
\end{equation}
\begin{equation}\label{equ:Phi_infty}
\frac{\partial \bmv}{\partial t}= - \tilde{\cV}_{\chi}(t) \bmv, ~x \in \oOmega,~t>0,
\end{equation}
and
\begin{equation}\label{equ:Phi_kappa}
\begin{cases}
\frac{\partial \bmv}{\partial t}= \bmkp \cL(x,t) \bmv - \cV_{\chi}(x,t) \bmv, &x \in \Omega, ~t>0,\\
\cB \bmv =0, & x \in \partial\Omega,~t>0,
\end{cases}
\end{equation}
respectively.
For any $x\in \oOmega$, let $\{\Gamma_{x,\chi}(t,s): t \geq s \}$ be the evolution family on $\Real^n$ of \eqref{equ:Phi_0}. Let  $\{\tilde{\Gamma}_{\chi}(t,s): t \geq s \}$ be the evolution family on $\Real^n$ of  \eqref{equ:Phi_infty}.
We assume that
\begin{enumerate}
	\item[(F)]  For any $\chi \in \cX$, $\cF_{\chi}(x,t)$ is nonnegative for all $(x,t) \in \oOmega \times \Real $.
	\item[(V)] For any $\chi \in \cX$, $-\cV_{\chi}(x,t)$ is cooperative for all $(x,t) \in \oOmega \times \Real$,   $\omega (\Gamma_{x,\chi})<0$ for all $x \in \oOmega$, and $\omega(\tilde{\Gamma}_{\chi})<0$.  Moreover, for any $\chi \in \cX$ and $\bmkp$ with all $\kappa_i >0$, $\omega (\Phi_{\bmkp,\chi})<0$.
\end{enumerate}

It follows from Proposition \ref{prop:eta} that $\omega(\Phi_{\chi})=\max_{ x \in \oOmega} \omega(\Gamma_{x,\chi})<0$ and $\omega(\tilde{\Phi}_{\chi})=\omega(\tilde{\Gamma}_{\chi})<0$. 
For any $\mu >0$, let $\{\bU_{\chi}^{\mu}(t,s): t \geq s \} $, $\{\tilde{\bU}_{\chi}^{\mu}(t,s): t \geq s \} $ and $\{\bU_{\bmkp,\chi}^{\mu}(t,s): t \geq s \} $ be the evolution families on $X$ of 
\begin{equation}\label{equ:U_0}
\frac{\partial \bmv}{\partial t}= - \cV_{\chi}(x,t) \bmv +\frac{1}{\mu} \cF_{\chi} (x,t) \bmv,~ x \in \oOmega,~ t>0,
\end{equation}
\begin{equation}\label{equ:U_infty}
\frac{\partial \bmv}{\partial t}= - \tilde{\cV}_{\chi}(t) \bmv +\frac{1}{\mu} \tilde{\cF}_{\chi} (t) \bmv, ~ x \in \oOmega,~t>0,
\end{equation}
and
\begin{equation}\label{equ:U_kappa}
\begin{cases}
\frac{\partial \bmv}{\partial t}= \bmkp \cL(x,t) \bmv - \cV_{\chi}(x,t) \bmv+ \frac{1}{\mu} \cF_{\chi} (x,t) \bmv, &x \in \Omega, ~t>0,\\
\cB \bmv =0, & x \in \partial\Omega,~t>0.
\end{cases}
\end{equation}
For any $\mu >0$ and $x \in \oOmega$, let $\{U_{x,\chi}^{\mu}(t,s): t\geq s\}$ be the evolution family on $\Real^n$ of \eqref{equ:U_0}.
For any $\mu >0$, let $\{\tilde{U}_{\chi}^{\mu}(t,s): t\geq s\}$ be the evolution family on $\Real^n$ of \eqref{equ:U_infty}.
We define a family of bounded linear positive operators $\tilde{\bF}_{\chi}(t)$
on $X$ by
$$
[\tilde{\bF}_{\chi}(t)\bm{\phi}](x):=\tilde{\cF}_{\chi}(t)[\bm{\phi}(x)], ~\forall (x,t) \in \oOmega \times \Real,~ \bm{\phi} \in X, 
$$
and  introduce three bounded linear positive operators $\bQ_{\chi} : \bX \rightarrow \bX$, $\tilde{\bQ}_{\chi} : \bX \rightarrow \bX$ and $\bQ_{\bmkp,\chi}: \bX \rightarrow \bX$ by
$$
[\bQ_{\chi}\bmu](t):= \int_{0}^{+\infty} \Phi_{\chi}(t,t-s)\cF_{\chi}(\cdot,t-s)\bmu(t-s) \D s, ~t \in \Real, ~\bmu \in \bX,
$$
$$
[\tilde{\bQ}_{\chi} \bmu] (t):= \int_{0}^{+\infty} \tilde{\Phi}_{\chi}(t,t-s)\tilde{\bF}_{\chi}(t-s) \bmu(t-s) \D s, ~ t \in \Real,~ \bmu \in \bX,
$$
and
$$
[\bQ_{\bmkp,\chi} \bmu](t): = \int_{0}^{+\infty} \Phi_{\bmkp,\chi}(t,t-s)\cF_{\chi}(\cdot,t-s)\bmu(t-s) \D s, ~ t \in \Real,~ \bmu \in \bX.
$$
Define $\R(\chi):= r(\bQ_{\chi})$, $\tilde{\cR}_0(\chi):=r(\tilde{\bQ}_{\chi})$ and $\R(\bmkp,\chi):=r(\bQ_{\bmkp,\chi})$.
Let
$$
\bP:=\{ \bmu \in C(\Real,\Real^n): \bmu(t)=\bmu(t+T), ~t \in \Real \}
$$
be a Banach space with the positive cone
$$
\bP_+:=\{ \bmu \in C(\Real,\Real_+^n): \bmu(t)=\bmu(t+T),~t \in \Real\}
$$
and the maximum norm $\Vert \bmu \Vert_{\bP} = \max_{1 \leq i \leq n}\max_{0 \leq t \leq T} \vert u_i (t) \vert$, where $\bmu=(u_1,u_2,\cdots,u_n)^T$. 
For each $x \in \oOmega$, we introduce a bounded linear positive operator $Q_{x,\chi} : \bP \rightarrow \bP$ by
$$
[Q_{x,\chi} \bmu] (t):=\int_{0}^{+\infty} \Gamma_{x,\chi}(t,t-s) \cF_{\chi}(x,t-s) \bmu(t-s) \D s, ~t \in \Real,~ \bmu \in \bP.
$$
Define a bounded linear positive operator $\tilde{Q}_{\chi} : \bP \rightarrow \bP$ by 
$$
[\tilde{Q}_{\chi} \bmu] (t):= \int_{0}^{+\infty} \tilde{\Gamma}_{\chi}(t,t-s)\tilde{\cF}_{\chi}(t-s)\bmu(t-s) \D s, ~ t \in \Real, ~\bmu \in \bP.
$$
Let $R_0(x,\chi):= r(Q_{x,\chi})$, $\forall x \in \oOmega$ and $\tilde{R}_{0}(\chi):=r(\tilde{Q}_{\chi})$.

\begin{lemma}\label{lem:A1A2}
	Assume that {\rm (F)} and {\rm (V)} hold, and $\cB$ satisfies the Dirichlet or Neumann or Robin boundary condition. Then {\rm (A1)} and {\rm (A2)} are satisfied, and hence, Theorem \ref{thm:R0:periodic} holds true,
	respectively,  for the following cases:
	\begin{itemize}
		\item[\rm (i)] $Y=X$, $\Phi=\Phi_{\bmkp,\chi}$, $F(t)=\cF_{\chi}(\cdot,t)$, $\R=\R(\bmkp,\chi)$, $\Psi_{\mu}=\bU_{\bmkp,\chi}^{\mu}$ with $\chi \in \cX$, $\kappa_i>0$ and $\mu>0$.
		\item[\rm (ii)] $Y=X$, $\Phi=\Phi_{\chi}$, $F(t)=\cF_{\chi}(\cdot,t)$, $\R=\R(\chi)$, $\Psi_{\mu}=\bU_{\chi}^{\mu}$ with $\chi \in \cX$ and $\mu>0$.
		\item[\rm (iii)]$Y=X$, $\Phi=\tilde{\Phi}_{\chi}$, $F(t)=\tilde{\bF}_{\chi}(t)$, $\R=\tilde{\cR}_0(\chi)$, $\Psi_{\mu}=\tilde{\bU}_{\chi}^{\mu}$ with $\chi \in \cX$ and $\mu>0$.
		\item[\rm (iv)] $Y=\Real^n$, $\Phi=\Gamma_{x,\chi}$, $F(t)=\cF_{\chi}(x,t)$, $\R=R_{0}(x,\chi)$, $\Psi_{\mu}=U_{x,\chi}^{\mu}$ with $\chi \in \cX$ and $\mu>0$ for all $x \in \oOmega$.
		\item[\rm (v)]$Y=\Real^n$, $\Phi=\tilde{\Gamma}_{\chi}$, $F(t)=\tilde{\cF}_{\chi}(t)$, $\R=\tilde{R}_0(\chi)$, $\Psi_{\mu}=\tilde{U}_{\chi}^{\mu}$ with $\chi \in \cX$ and  $\mu>0$.
	\end{itemize}
\end{lemma}
We remark that the conclusions of Theorem \ref{thm:R0:periodic} for the case (i), (ii) and (iii), and (iv) and (v) in Lemma \ref{lem:A1A2} can also be derived 
from \cite[Theorems 3.7 and 3.8 and Proposition 3.9 with $\tau=0$]{liang2019basic}, \cite[Proposition 3.6(ii)]{liang2019principal} and Lemma \ref{lem:R0:equiv}, and \cite[Theorems 2.1 and 2.2]{wang2008threshold}, respectively.
\begin{lemma}\label{lem:0_infty}
	Assume that {\rm (F)} and {\rm (V)} hold. Then the following statements are valid for any $\chi \in \cX$:
	\begin{itemize}
		\item[\rm (i)] 	$\R(\chi)=\max_{ x \in \oOmega} R_0 (x,\chi)$.
		\item[\rm (ii)] $\tilde{\cR}_0(\chi)=\tilde{R}_{0}(\chi)$.
	\end{itemize}
\end{lemma}
\begin{proof}
	We only need to prove (i), since (ii) can be derived in a similar way.
	Fix an $\chi \in \cX$ and let $\overline{\cR}_0 := \max_{ x \in \oOmega} R_0(x,\chi)$. Below we proceed with  two cases.
	
	In the case where $\overline{\cR}_0>0$, there exists some $x_0 \in \oOmega$ such that $R_0(x_0,\chi)=\overline{\cR}_0>0$. By Theorem \ref{thm:R0:periodic}(ii) and Lemma \ref{lem:A1A2}(iv), we have $\omega(U_{x_0,\chi}^{\overline{\cR}_0})=0$. It follows from $\overline{\cR}_0 \geq R_0(x,\chi)$, $\forall x \in \oOmega$ and Corollary \ref{cor:periodic:R0}(i) and Lemma \ref{lem:A1A2}(iv) that $\omega(U_{x,\chi}^{\overline{\cR}_0}) \leq 0$, $\forall x \in \oOmega$, and hence, $\max_{ x \in \oOmega} \omega(U_{x,\chi}^{\overline{\cR}_0})=0$. By Proposition \ref{prop:eta}, we have $\omega(\bU_{\chi}^{\overline{\cR}_0})=\max_{ x \in \oOmega} \omega(U_{x,\chi}^{\overline{\cR}_0})=0$. Now Theorem \ref{thm:R0:periodic}(ii), Corollary \ref{cor:periodic:R0}(ii) and Lemma \ref{lem:A1A2}(ii) yield that $\R(\chi)=\overline{\cR}_0 >0$.
	
	In the case where $\overline{\cR}_0=0$, we have $R_0(x,\chi)=0$, $\forall x\in \oOmega$. Corollary \ref{cor:periodic:R0}(ii) and Lemma \ref{lem:A1A2}(iv) imply that $\omega(U_{x,\chi}^{\mu})<0$ for all $x\in \oOmega$, $\mu >0$. By Proposition \ref{prop:eta} again, we have $\omega(\bU_{\chi}^{\mu})=\max_{ x \in \oOmega} \omega(U_{x,\chi}^{\mu})<0$ for all $\mu >0$. Therefore, $\R(\chi)=0$ due to Corollary \ref{cor:periodic:R0}(ii) and Lemma \ref{lem:A1A2}(ii).
\end{proof}

Now we are ready to prove the main result of this paper.

\begin{theorem}\label{thm:R0:infty:Neumann}
	Assume that {\rm (F)} and {\rm (V)} hold, and there exists $\chi_0 \in \cX$
	such that  
	$\cV_{\chi}$ and $\cF_{\chi}$ converge uniformly to $\cV_{\chi_0}$ and $\cF_{\chi_0}$ as $\chi \rightarrow \chi_0$, respectively. We use $\R^{\cD}(\bmkp,\chi)$, $\R^{\cN}(\bmkp,\chi)$ and $\R^{\cR}(\bmkp,\chi)$ to denote the special cases of $\R(\bmkp,\chi)$ when $\cB$ represents the Dirichlet, Neumann and Robin boundary conditions, respectively. Then the following statements are vaild:
	\begin{itemize}
		\item[\rm (i)] 
		$\lim\limits_{\max_{1 \leq i \leq n} \kappa_i \rightarrow 0,\chi \rightarrow \chi_0} \R(\bmkp,\chi)= \max_{ x \in \oOmega} R_0 (x,\chi_0)$. 
		\item[\rm (ii)]
		$\lim\limits_{\min_{1 \leq i \leq n} \kappa_i \rightarrow +\infty,\chi \rightarrow \chi_0} \R^{\cD}(\bmkp,\chi)=
		\lim\limits_{\min_{1 \leq i \leq n} \kappa_i \rightarrow +\infty,\chi \rightarrow \chi_0} \R^{\cR}(\bmkp,\chi)= 0.$
		\item[\rm (iii)]
		$\lim\limits_{\min_{1 \leq i \leq n} \kappa_i \rightarrow +\infty,\chi \rightarrow \chi_0} \R^{\cN}(\bmkp,\chi)= \tilde{R}_{0}(\chi_0).$
	\end{itemize}
\end{theorem}
\begin{proof}
	(i) We only prove $$\lim\limits_{\max_{1 \leq i \leq n} \kappa_i \rightarrow 0,\chi \rightarrow \chi_0} \R(\bmkp,\chi)= \max_{ x \in \oOmega} R_0 (x,\chi_0),$$ since the other two limiting profiles can be obtained in a similar way. In view of Lemma \ref{lem:0_infty} (i), 
	 it suffices to prove 
	\begin{equation}\label{equ:R0:0:N}
	\lim\limits_{\max_{1 \leq i \leq n} \kappa_i \rightarrow 0,\chi \rightarrow \chi_0}\R(\bmkp,\chi)=\R(\chi_0),
	\end{equation}
	 Here we choose  $\Theta=(\Int(\Real^n_+) \times \cX)\cup \{\theta_0\}$, and let   $\theta_0:=(0,0,\cdots,0,\chi_0)\in \Theta$ and 
	  $\theta:=( \kappa_1,\kappa_2,\cdots,\kappa_n,\chi)\in \Theta$.
	   In order  to obtain \eqref{equ:R0:0:N},  , it suffices to prove the following claim.
	
	{\it Claim.} For any $ \mu >0$, $\lim\limits_{\max_{1 \leq i \leq n,} \kappa_i \rightarrow 0,\chi \rightarrow \chi_0}\omega(\bU_{\bmkp,\chi}^{\mu}) = \omega(\bU^{\mu}_{\chi_0}) $.
	
	Without loss of generality, we prove the claim under the assumption $\mu=1$. For any $\chi \in \cX$, let $\cM_{\chi}=(m_{ij,\chi})_{n\times n}=-\cV_{\chi} + \cF_{\chi}$. For any given $\delta>0$, let $\overline{\cM}_{\chi_0}^{\delta}=(\overline{m}_{ij,\chi_0}^{\delta})_{n\times n}$ and $\underline{\cM}_{\chi_0}^{\delta}=(\underline{m}_{ij,\chi_0}^{\delta})_{n\times n}$, where
	$$
	\overline{m}_{ij,\chi_0}^{\delta}(x,t)= m_{ij,\chi_0}(x,t)+\delta, \text{ and }
	\underline{m}_{ij,\chi_0}^{\delta}(x,t)=
	\begin{cases}
	m_{ij,\chi_0}(x,t) - \delta& i=j,\\
	\max(0,m_{ij,\chi_0}(x,t) - \delta), & i \neq j.
	\end{cases}
	$$
	Thus, $\cM_{\chi}(x,t)$, $\cM_{\chi_0}(x,t)$, $\overline{\cM}_{\chi_0}^{\delta}(x,t)$ and $\underline{\cM}_{\chi_0}^{\delta}(x,t)$ are cooperative for all $(x,t) \in \oOmega \times \Real$. Let $\{\underline{\bU}_{\chi_0}^{\delta}(t,s): t \geq s \} $ and $\{\overline{\bU}_{\chi_0}^{\delta}(t,s): t \geq s \} $ be the evolution families on $X$ of 
	$$
	\frac{\partial \bmv}{\partial t}= \underline{\cM}_{\chi_0}^{\delta}(x,t) \bmv, ~ x \in \oOmega,~t>0,
	$$
	and
	$$
	\frac{\partial \bmv}{\partial t}= \overline{\cM}_{\chi_0}^{\delta}(x,t) \bmv, ~ x \in \oOmega,~t>0,
	$$
	respectively.
	
	By Proposition \ref{prop:eta} and the perturbation theory of matrix (see, e.g., \cite[Section II.5.1]{kato1976perturbation}), it follows that for any given $\epsilon>0$, there exists $\delta_0>0$ such that
	$$
	\omega(\bU^{1}_{\chi_0}) -\frac{\epsilon}{2} 
	\leq \omega(\underline{\bU}_{\chi_0}^{\delta}), 
	\text{ and }
	\omega(\overline{\bU}_{\chi_0}^{\delta}) 
	\leq \omega(\bU^{1}_{\chi_0})+\frac{\epsilon}{2},
	$$ for all $\delta \leq \delta_0$. Moreover, 
	$\underline{\cM}_{\chi_0}^{\delta_0} (x,t)\leq \cM_{\chi}(x,t) \leq \overline{\cM}_{\chi_0}^{\delta_0} (x,t)$, $\forall (x,t) \in \oOmega \times \Real$ provided that $\rho(\chi,\chi_0) \leq \hat{\sigma}$ for some $\hat{\sigma}>0$.
	Let $\{\underline{\bU}_{\bmkp,\chi_0}^{\delta_0}(t,s): t \geq s \}$ and $\{\overline{\bU}_{\bmkp,\chi_0}^{\delta_0}(t,s): t \geq s \}$ be the evolution families on $X$ of 
	$$
	\begin{cases}
	\frac{\partial \bmv}{\partial t}= \bmkp \cL(x,t) \bmv +\underline{\cM}_{\chi_0}^{\delta_0}(x,t) \bmv, &~ x \in \Omega,~t>0,\\
	\cB \bmv =0, & ~ x \in \partial\Omega,~t>0,
	\end{cases}
	$$
	and
	$$
	\begin{cases}
	\frac{\partial \bmv}{\partial t}= \bmkp \cL(x,t) \bmv + \overline{\cM}_{\chi_0}^{\delta_0}(x,t)\bmv, &~ x \in \Omega,~t>0,\\
	\cB \bmv =0, & ~ x \in \partial\Omega,~t>0,
	\end{cases}
	$$
	respectively.
	In view of Proposition \ref{prop:eig:com}, we have
	$$
	 \omega({\underline{\bU}}_{\bmkp,\chi_0}^{\delta_0}) 
	\leq \omega(\bU^{1}_{\bmkp,\chi}) 
	\leq \omega(\overline{\bU}_{\bmkp,\chi_0}^{\delta_0}) 
	$$
	 provided that $\rho(\chi,\chi_0) \leq \hat{\sigma}$. According to Theorem \ref{thm:0}, there is  $\hat{\kappa}>0$ such that 
	$$
	\omega({\underline{\bU}}_{\chi_0}^{\delta_0}) - \frac{\epsilon}{2} 
	\leq \omega({\underline{\bU}}_{\bmkp,\chi_0}^{\delta_0}), \text{ and }
	 \omega(\overline{\bU}_{\bmkp,\chi_0}^{\delta_0}) 
	\leq \omega(\overline{\bU}_{\chi_0}^{\delta_0}) +\frac{\epsilon}{2},
	$$ 
	provided that $\max_{1 \leq i \leq n} \kappa_i \leq \hat{\kappa}$. One can easily verify that
	$$
	\omega(\bU^{1}_{\chi_0})-\epsilon
	\leq\omega({\underline{\bU}}_{\chi_0}^{\delta_0}) - \frac{\epsilon}{2} 
	\leq \omega(\bU^{1}_{\bmkp,\chi}) 
	\leq \omega(\overline{\bU}_{\chi_0}^{\delta_0}) +\frac{\epsilon}{2} 
	\leq \omega(\bU^{1}_{\chi_0}) +\epsilon,
	$$
	whenever $\max_{1 \leq i \leq n} \kappa_i \leq \hat{\kappa}$ and $\rho(\chi,\chi_0) \leq \hat{\sigma}$.  This  proves the claim above.  It then 
	follows that  \eqref{equ:R0:0:N} holds true, and hence,  statement  (i) is valid.
	
	(ii) We use $\bU_{\bmkp,\chi}^{\mu,\cD}$, $\bU_{\bmkp,\chi}^{\mu,\cN}$ and $\bU_{\bmkp,\chi}^{\mu,\cR}$ to denote the special case of $\bU_{\bmkp,\chi}^{\mu}$ when $\cB$ represents the Dirichlet, Neumann and Robin boundary conditions, respectively. By repeating the arguments  in the proof of Claim, we see that for any $ \mu >0$, 
	$$
	\lim\limits_{\min_{1 \leq i \leq n} \kappa_i \rightarrow +\infty,\chi \rightarrow \chi_0}\omega(\bU_{\bmkp,\chi}^{\mu,\cD})
	=\lim\limits_{\min_{1 \leq i \leq n} \kappa_i \rightarrow +\infty,\chi \rightarrow \chi_0}\omega(\bU_{\bmkp,\chi}^{\mu,\cR})  
	= -\infty,
	$$  
	that is,
	$$
	\lim\limits_{\max_{1 \leq i \leq n} \frac{1}{\kappa_i} \rightarrow 0,\chi \rightarrow \chi_0}\omega(\bU_{\bmkp,\chi}^{\mu,\cD})
	=\lim\limits_{\max_{1 \leq i \leq n} \frac{1}{\kappa_i} \rightarrow 0,\chi \rightarrow \chi_0}\omega(\bU_{\bmkp,\chi}^{\mu,\cR})  
	= -\infty,
	$$ 
	We let $\alpha_i :=\frac{1}{\kappa_i}$, $\forall i=1,2,\cdots,n$, $\kappa_i>0$.
	Choosing  $\Theta=(\Int(\Real^n_+) \times \cX) \cup \{ \theta_0 \}$, $\theta:=(\alpha_1,\alpha_2,\cdots,\alpha_n, \chi)$, and $\theta_0:=(0,0,\cdots,0,\chi_0)$,
	we then obtain the desired conclusion due to Theorem \ref{thm:R0:periodic:conti:0} and Lemmas \ref{lem:A1A2}(i). 
	
	(iii) In view of Lemma \ref{lem:0_infty}(i), it suffices to prove 
		\begin{equation}\label{equ:R0:infty:N}
	\lim\limits_{\min_{1 \leq i \leq n} \kappa_i \rightarrow +\infty,\chi \rightarrow \chi_0}\R(\bmkp,\chi)=\tilde{\cR}_0(\chi_0).
	\end{equation}
	By the arguments similar to those in the proof of  the claim above, it follows  that for any $ \mu >0$, 
	$$
	\lim\limits_{\min_{1 \leq i \leq n,} \kappa_i \rightarrow +\infty,\chi \rightarrow \chi_0}\omega(\bU_{\bmkp,\chi}^{\mu,\cN}) = \omega(\tilde{\bU}^{\mu}_{\chi_0}),
	$$ 
	that is,
	$$
	\lim\limits_{\max_{1 \leq i \leq n} \frac{1}{\kappa_i} \rightarrow 0,\chi \rightarrow \chi_0}\omega(\bU_{\bmkp,\chi}^{\mu,\cN}) = \omega(\tilde{\bU}^{\mu}_{\chi_0}),
	$$ 
	We let $\alpha_i :=\frac{1}{\kappa_i}$, $\forall i=1,2,\cdots,n$, $\kappa_i>0$, and choose $\Theta=(\Int(\Real^n_+) \times \cX) \cup \{ \theta_0 \}$, $\theta:=(\alpha_1,\alpha_2,\cdots,\alpha_n, \chi)$, and $\theta_0:=(0,0,\cdots,0,\chi_0)$.
	Thus, \eqref{equ:R0:infty:N} follows from Theorem \ref{thm:R0:periodic:conti} and Lemma \ref{lem:A1A2}(i) and (iii).
\end{proof}

\section{Periodic solutions}\label{sec:nonlinear}
We use the same notations $\Omega$, $n$, $X$, $X_+$, $\bX$, $\bX_+$, $\bmkp$, $\cL$ and $\cB$ as in section \ref{sec:AB}, and consider the time-periodic reaction-diffusion system
\begin{equation}\label{equ:sys:main:nonlinear}
\begin{cases}
\frac{\partial \bmw}{\partial t}= \bmkp \cL(x,t) \bmw + \cG(x,t,\bmw),  &x\in \Omega,~t>0,\\
\cB \bmw=0, & x\in \partial \Omega,~t>0,
\end{cases}
\end{equation}
where the reaction term satisfies 
$$\cG(x,t,q_1,q_2,\cdots,q_n)=(G_1,G_2,\cdots,G_n)^T(x,t,q_1,q_2,\cdots,q_n)\in C^1(\oOmega \times \Real \times \Real^n_+,\Real^n),$$
and $G_i(x,t,q_1,q_2,\cdots,q_n) =G_i(x,t+T,q_1,q_2,\cdots,q_n)$ for some $T>0$. 

Let 
$$
\tilde{G}_i(t,q_1,q_2,\cdots,q_n)=\frac{1}{\vert \Omega \vert} \int_{\Omega} G_i(x,t,q_1,q_2,\cdots,q_n) \D x,
$$
and
$$
\tilde{\cG}(t,q_1,q_2,\cdots,q_n)=(\tilde{G}_1,\tilde{G}_2,\cdots,\tilde{G}_n)^T(t,q_1,q_2,\cdots,q_n).
$$ 
We assume that
\begin{itemize}
	\item[(H1)] For any  $1 \leq i \neq j \leq n$, $\frac{\partial}{\partial q_i } G_j (x,t,q_1,q_2,\cdots,q_n) \geq 0$ for all $(x,t) \in \oOmega \times \Real$.
	\item[(H2)] For each $x \in \oOmega$, the ODE system  $$\bmw'(t)= \cG(x,t,\bmw(t)), ~\bmw(0) \in \Real_+^n \setminus \{ 0 \}$$ has a globally asymptotically stable positive $T$-periodic solution, denoted by 
	$\bmw_{0}(x,t)=(w_{0,1},w_{0,2},\cdots,w_{0,n})^T(x,t)$. 
	Moreover, $\bmw_{0}$ is continuous on $\oOmega \times \Real$, and  the ODE system  $$\bmw'(t)= \tilde{\cG}(t,\bmw(t)), ~\bmw(0) \in \Real_+^n \setminus \{ 0 \}$$ has a globally asymptotically stable positive $T$-periodic solution, denoted by 
	$\tilde{\bmw}_{\infty}(t)=(\tilde{w}_{\infty,1},\tilde{w}_{\infty,2},\cdots,\tilde{w}_{\infty,n})^T(t)$. 
	\item[(H3)] $G_i (x,t,0)=0$ on $\oOmega \times \Real$ for all $1\leq i \leq n$, and there exists $\underline{\bmv}(t)=(\underline{v}_1,\underline{v}_2,\cdots,\underline{v}_n)^T(t)$ with each $\underline{v}_i(t)>0$ and $\underline{v}_i(t)=\underline{v}_i(t+T)$ for all $t \in \Real$  such that the following inequality holds for any $\tau \in[0,1]$:
	$$
	\tau \underline{\bmv}'(t) \leq  \cG(x,t,\tau \underline{\bmv}^T(t)), ~\forall 
	(x,t) \in \oOmega \times \Real.
	$$
	\item[(H4)] There exist $h>0$ and $\overline{\bm{v}}=(\overline{v}_1, \overline{v}_2, \cdots, \overline{v}_n)^T \in \Real^n$ with each $\overline{v}_i >0 $ such that 
	$$
	\sup_{\tau \rightarrow \infty, (x,t) \in \oOmega \times [0,T]} G_i (x,t,\tau \overline{\bmv}) \leq  -h, \, \,  \forall 1\leq i \leq n.
	$$
	\item[(H5)] For any $\bmkp$ with $\kappa_i>0$, $\bmw_{\bmkp}$ is the unique positive periodic solution of system \eqref{equ:sys:main:nonlinear}.
\end{itemize}

By the comparison arguments, we have the following observation.

\begin{lemma}\label{lem:nonlinear:exist}
	Assume that {\rm (H1)--(H4)} hold. Then  system \eqref{equ:sys:main:nonlinear} has at least one positive periodic solution. Furthermore, any positive periodic solution $\bmw=(w_1,w_2,\cdots,w_n)^T$ of \eqref{equ:sys:main:nonlinear}  satisfies $\tau_1\underline{v}_j(t) \leq  w_j \leq  \tau_2 \overline{v}_j$ in $\oOmega \times \Real$ for some $\tau_1 \in (0,1]$ and $\tau_2>0$, where $\underline{v}_j$ and $\overline{v}_j$ are given in {\rm (H3)} and {\rm (H4)}.
\end{lemma}

\begin{proof}
We first choose a real number $\tau_2>0$ such that
	$G_i (x,t,\tau_2 \overline{\bmv}) \leq  -\frac{h}{2}<0.$
It then follows that 
	$$
	(\tau_2 \overline{v}_i)'=0 > G_i (x,t,\tau_2 \overline{\bmv}),~ \forall (x,t) \in \Omega \times \Real,~i =1,2,\cdots,n.
	$$
	On the other hand, there exists $\tau_1>0$ small enough such that $\tau_1 \underline{v}_i(t) < \tau_2 \overline{v}_i$,  $\forall t\in \Real, ~ i =1,2,\cdots,n$. The desired conclusion follows from  the standard comparison arguments and iteration method (see, e.g., \cite[Lemma 5.3]{bai2020asymptotic}).
\end{proof}

By the arguments similar to those for \cite[Theorem 1.7]{bai2020asymptotic} and \cite[Theorem 1.5]{lam2016asymptotic}, we have the following observation.
\begin{theorem}\label{thm:nonlinear:convergence:0}
	Assume that {\rm (H1)--(H5)} hold. Then $\bmw_{\bmkp} (x,t) \rightarrow \bmw_{0}(x,t)$ uniformly on $\oOmega \times \Real$ as $\max_{1 \leq i \leq n}  \kappa_i  \rightarrow 0$.
\end{theorem}

Next we turn to the asymptotic behavior of the positive periodic solution with large diffusion coefficients.
Let $\tau_2>0$ be given in Lemma \ref{lem:nonlinear:exist} and define a set
$$
\cQ:=\{(q_1,q_2,\cdots,q_n) : 0 \leq q_i \leq \tau_2 \overline{v}_i,~ 1 \leq i \leq n\}.
$$
Let $\tilde{\bmw}_{\bmkp}:=(\tilde{w}_{\bmkp,1},\tilde{w}_{\bmkp,2},\cdots,\tilde{w}_{\bmkp,n})^T$ and $\hat{\bmw}_{\bmkp}:=(\hat{w}_{\bmkp,1},\hat{w}_{\bmkp,2},\cdots,\hat{w}_{\bmkp,n})^T$, where 
\begin{equation}\label{equ:tilde:hat:w}
\tilde{w}_{\bmkp,i} (t) = \frac{1}{\vert \Omega \vert} \int_{\Omega} w_{\bmkp,i}(y,t) \D y, ~\hat{w}_{\bmkp,i} (x,t) =w_{\bmkp,i}(x,t) - \tilde{w}_{\bmkp,i} (t), ~ 1 \leq i \leq n.
\end{equation}
Then we have the following observation.
\begin{lemma}\label{lem:nonlinear:convergence:infty:tilde}
	Assume that {\rm (H1)--(H5)} hold. Then $\tilde{\bmw}_{\bmkp} (t) \rightarrow \tilde{\bmw}_{\infty}(t)$ uniformly on $ \Real$ as $\min_{1 \leq i \leq n}  \kappa_i  \rightarrow +\infty$.
\end{lemma}
\begin{proof}
	Taking the average of $i$-th equation of  \eqref{equ:sys:main:nonlinear} over $\Omega$, we have
	\begin{equation}\label{equ:w:tilde:diff}
	\frac{\D}{\D t}\tilde{w}_{\bmkp,i}(t)=
	\frac{1}{\vert \Omega \vert }\int_{\Omega} G_i(x,t,\bmw_{\bmkp}^T(x,t)) \D x,~ \forall t \in \Real.
	\end{equation}
	Clearly, there exists $C_0>0$ independent of $i$ and $\bmkp$ such that 
	$$
	\vert G_i (x,t,q_1,q_2,\cdots,q_n) \vert \leq C_0
	\text{ on } \oOmega \times \Real \times \cQ,~ i=1,2,\cdots,n.
	$$
	It then easily follows  that
	\begin{equation}\label{equ:w:tilde:diff:estimate}
	\left\vert \frac{\D}{\D t}\tilde{w}_{\bmkp,i}(t) \right\vert \leq C_0, ~ \forall t \in \Real, ~ i=1,2,\cdots,n.
	\end{equation}
	On the other hand, we integrate  \eqref{equ:w:tilde:diff}  from $0$ to $t$ to obtain
	$$
	\tilde{w}_{\bmkp,i}(t)-\tilde{w}_{\bmkp,i}(0)=
	\frac{1}{\vert \Omega \vert }\int_{0}^{t}\int_{\Omega} G_i(x,s,\bmw_{\bmkp}^T(x,s)) \D x \D s, ~\forall t \in [0,T].
	$$
	For each $i=1,2,\cdots,n$, we define $$g_{\bmkp,i}(t):=\frac{1}{\vert \Omega \vert}\int_{0}^{t} \int_{\Omega} 
	[G_i(x,s,\bmw_{\bmkp}^T(x,s)) - G_i(x,s,\tilde{\bmw}_{\bmkp}^T(s)) ] \D x \D s, ~\forall t \in \Real
	$$
	 to obtain 
	$$
	\begin{aligned}
	\tilde{w}_{\bmkp,i}(t)-\tilde{w}_{\bmkp,i}(0)
	&=
	\frac{1}{\vert \Omega \vert }\int_{0}^{t}\int_{\Omega} G_i(x,s,\tilde{\bmw}_{\bmkp}^T(s)) \D x \D s+ g_{\bmkp,i}(t)\\
	&=\int_{0}^{t} \tilde{G}_i (s,\tilde{\bmw}_{\bmkp}^T(s))  \D s+ g_{\bmkp,i}(t),~ \forall t \in [0,T],~i=1,2,\cdots,n.
	\end{aligned}
	$$
	Letting $\underline{\kappa}:=\min_{1 \leq i \leq n} \kappa_i$, we then have the following claim.
	
	{\it Claim 1.} For each $i=1,2,\cdots,n$, $g_{\bmkp,i}(t) \rightarrow 0$ uniformly on $[0,T]$ as $\underline{\kappa} \rightarrow +\infty$.
	
	Let us postpone the proof of Claim 1,  and reach the conclusion quickly.
	Let 
	$
	\tilde{\bX} :=C([0,T],\Real^n)
	$
	be a Banach space with the maximum norm 
	$
	\Vert \bmu \Vert_{\tilde{\bX}} = \max_{1 \leq i \leq n} \max_{ 0 \leq t \leq T} \vert u_i(t) \vert,
	$ where $\bmu=(u_1,u_2,\cdots,u_n)^T \in \tilde{\bX}$. Suppose, by contradiction, that  there is a sequence  $\bmkp_l=(\kappa_{l,1},\kappa_{l,2},\cdots,\kappa_{l,n})^T$ with
	$\min_{1 \leq i \leq n} \kappa_{l,i} \rightarrow +\infty$ as $l \rightarrow \infty$ such that $\Vert \tilde{\bmw}_{\bmkp_{l}} - \tilde{\bmw}_{\infty} \Vert_{\tilde{\bX}} \geq \epsilon_0$ for some $\epsilon_0 >0$ for all $\bmkp_l$ with $\kappa_{l,i}>0$. By the Ascoli–Arzel\`{a} theorem (see, e.g., \cite[Theorem I.28]{reed1980methods}), together with Lemma \ref{lem:nonlinear:exist} and the estimate \eqref{equ:w:tilde:diff:estimate}, it follows that there exists a subsequence  $\bmkp_{l_k}=(\kappa_{l_k,1},\kappa_{l_k,2},\cdots,\kappa_{l_k,n})^T$
	of  $\bmkp_l$ such that 
	$$
	\tilde{\bmw}_{\bmkp_{l_k}} \rightarrow \bmv=(v_{1},v_{2},\cdots,v_{n})^T \text{  in } \tilde{\bX} \text{  as }\min_{1 \leq i \leq n} \kappa_{l_k,i} \rightarrow +\infty
	$$ 
	and $\bmv(0)=\bmv(T)$. Moreover,
	\begin{equation}\label{equ:w:tilde:i}
	v_{i}(t)-v_{i}(0)
	=\int_{0}^{t} \tilde{G}_i (s,\bmv^T(s)) \D s,~ \forall t \in [0,T],~ 1 \leq i \leq n,
	\end{equation}
	that is,
	\begin{equation}\label{equ:w:tilde}
	\bmv(t)-\bmv(0)
	=\int_{0}^{t} \tilde{\cG} (s,\bmv^T(s)) \D s,~ \forall t \in [0,T].
	\end{equation}
	Thanks to (H2) and \eqref{equ:w:tilde}, we obtain $\bmv(t)=\tilde{\bmw}_{\infty}(t)$, $\forall t\in [0,T]$, a contradiction.
	
	In order to prove Claim 1, we  need  the following claim.
	
	{\it Claim 2.} There exists $C_1 >0$ dependent on $\underline{a}$, $T$, $\Omega$, $C_0$, $\tau_2$ and $\max_{1 \leq j \leq n} \overline{v}_j$ such that
	\begin{equation}\label{equ:nabla:w}
	\int_{0}^{T}\int_{\Omega} \vert \nabla w_{\bmkp,i} \vert^2 \D x \D t
	\leq C_1\kappa_i^{-1}.
	\end{equation}
	
Indeed, Lemma \ref{lem:nonlinear:exist} implies that $\bmw_{\bmkp}(x,t)=(w_{\bmkp,1},w_{\bmkp,2},\cdots,w_{\bmkp,n})^T(x,t) \in \cQ$ for all $\bmkp$ with $\kappa_i>0$.
	Choose $\underline{a} >0$ such that $\sum_{p,q=1}^{N} a_{pq}^{i} \xi_p \xi_q \geq \underline{a} \sum_{p=1}^{N} \xi_p^2$, $ \forall 1 \leq i \leq n$. We multiply the $i$-th equation of  \eqref{equ:sys:main:nonlinear} by $w_{\bmkp,i}$ and then integrate over $\Omega \times (0,T)$ to obtain
	$$
	\int_{0}^{T}\int_{\Omega} \kappa_i \sum_{p,q=1}^{N} a_{pq}^{i} (w_{\bmkp,i})_{x_p} (w_{\bmkp,i})_{x_q} \D x \D t
	= \int_{0}^{T}\int_{\Omega} G_i(x,t,\bmw_{\bmkp}^T) w_{\bmkp,i} \D x \D t.
	$$
	It follows that 
	$$
	\underline{a} \kappa_i\int_{0}^{T}\int_{\Omega}  \vert \nabla w_{\bmkp,i} \vert^2 \D x \D t \leq T \vert \Omega \vert C_0 \tau_2 \max_{1 \leq j \leq n} \overline{v}_j,   \quad \forall \, 1\leq i\leq n.
	$$
	Thus,  Claim 2 holds true.  Now let  us return to the proof of Claim 1. It is easy to see that
	$
	\int_{\Omega} \hat{w}_{\bmkp,i} (x,t) \D x =0
	$, $\forall t \in \Real,~  i=1,2,\cdots,n$. By Poincar\'e's inequality (see, e.g., \cite[Theorem 5.8.1]{evans2010partial}), we have
	$$
	\int_{\Omega} \vert \hat{w}_{\bmkp,i} \vert^2 \D x \leq C_2 	\int_{\Omega} \vert \nabla \hat{w}_{\bmkp,i} \vert^2 \D x,~ \forall 1 \leq i \leq n,~ t \in \Real,
	$$
	for some $C_2$ dependent on $\Omega$. Therefore,  it follows from $\nabla \hat{w}_{\bmkp,i} = \nabla w_{\bmkp,i}$ and \eqref{equ:nabla:w}
	that
	$$
	\int_{0}^{T}\int_{\Omega} \vert \hat{w}_{\bmkp,i} \vert^2 \D x \D t
	\leq C_1 C_2\kappa_i^{-1},~ 1 \leq i \leq n.
	$$
	Thus, the Caughy inequality yields
	$$
	\int_{0}^{T}\int_{\Omega} \vert \hat{w}_{\bmkp,i} \vert \D x \D t
	\leq C_3\kappa_i^{-\frac{1}{2}},~  1 \leq i \leq n,
	$$
	for some constant $C_3$ dependent on $\underline{a}$, $T$, $\Omega$, $C_0$, $\tau_2$ and $\max_{1 \leq j \leq n} \overline{v}_j$.
	On the other hand, since $\bmw_{\bmkp}^T(x,t)\in \cQ$ and $\tilde{\bmw}_{\bmkp}^T(x,t) \in \cQ$, $\forall (x,t)\in \oOmega \times \Real$, there exists $\hat{L}$ independent of $x$, $t$, $\bmkp$ and $i$ such that
	$$
	\vert
	G_i(x,t,\bmw_{\bmkp}^T(x,t)) - G_i(x,t,\tilde{\bmw}_{\bmkp}^T(t)) 
	\vert \leq \hat{L} \max_{1 \leq i \leq n} \vert w_{\bmkp,i}(x,t) - \tilde{w}_{\bmkp,i}(t) \vert, ~\forall (x,t)\in \oOmega \times \Real.
	$$
	By a straightforward computation, we have
	$$
	\begin{aligned}
	&\int_{0}^{t} \int_{\Omega}
	\left\vert 
	G_i(x,s,\bmw_{\bmkp}^T(x,s)) - G_i(x,s,\tilde{\bmw}_{\bmkp}^T(s)) 
	\right\vert
	\D x \D s \\
	\leq&  \max_{1 \leq i \leq n} \hat{L} \int_{0}^{t} \int_{\Omega} \vert w_{\bmkp,i}(x,s) - \tilde{w}_{\bmkp,i}(s) \vert\D x \D s\\
	\leq&  \max_{1 \leq i \leq n} \hat{L} \int_{0}^{t} \int_{\Omega} \vert \hat{w}_{\bmkp,i}(x,s)  \vert\D x \D s\\
	\leq& \hat{L} C_3 \underline{\kappa}^{-\frac{1}{2}}, ~\forall t \in [0,T].\\
	\end{aligned}
	$$
	This implies that $g_{\bmkp,i}(t) \rightarrow 0$ uniformly on $[0,T]$ as $ \underline{\kappa} \rightarrow +\infty$. 
\end{proof}

\begin{lemma} \label{lem:eig:L_i}
	For each $j=1,2,\cdots,n$, we consider the following eigenvalue problem:
	\begin{equation}\label{equ:eig:i:kappa}
	\begin{cases}
	\frac{\partial u}{\partial t}= \kappa_j \cL_j (x,t) u + \lambda u, &(x,t) \in \Omega \times \Real,\\
	\cB_j u=0, & (x,t) \in \partial \Omega\times \Real.
	\end{cases}
	\end{equation}
	Here, $\cB_j$ is the Neumann boundary condition given in \eqref{equ:Neumann}.
	Let $\{ \lambda_{j,m} \}_{m=1}^{\infty}$ be all eigenvalues of \eqref{equ:eig:i:kappa} such that $\Realparts \lambda_{j,m}$ is non-decreasing with respect to  $m$.
	Then the following statements are valid:
	\begin{itemize}
		\item[\rm (i)] The principal eigenvalue $\lambda_{j,1}$ of \eqref{equ:eig:i:kappa} is zero and simple.
		\item[\rm (ii)]  There exists $\gamma_0>0$, independent of $j$, such that $\Realparts \lambda_{j,m} > \gamma_0 \kappa_j$ for all $1\leq j\leq n$ and $m \geq 2$.
	\end{itemize}
\end{lemma}

\begin{proof}
	Statement (i) can be derived by the Krein-Rutman theorem (see, e.g., \cite[Section 7]{hess1991periodic}).
	
	(ii) Choose $\underline{a} >0$ independent of $j$ such that $\sum_{p,q=1}^{N} a_{pq}^{j} \xi_p \xi_q \geq \underline{a} \sum_{p=1}^{N} \xi_p^2$. For any given $j=1,2,\cdots,n$, let $\lambda= \lambda_{R} +\bmi \lambda_{I}$, $u=u_{R}+ \bmi u_{I}$ be an eigenpair of \eqref{equ:eig:i:kappa} with $\lambda \neq 0$, where $\bmi^2+1=0$. By Definition \ref{def:principal}, we obtain $\lambda_{R}>0$. Then we have the following claims.
	
	{\it Claim 1.} For any $ (x,t) \in \partial \Omega\times \Real$, $\cB_j u_{R}=\cB_j u_{I}=0$.
	
	{\it Claim 2.} For any $t\in \Real$, $\tilde{u}_R(t):=\int_{\Omega} u_R (x,t) \D x=0$, $\tilde{u}_I(t):=\int_{\Omega} u_I (x,t) \D x=0$.
	
	{\it Claim 3.} There holds
	\begin{equation}\label{equ:int:lambda_R}
	-\int_{0}^{T}\int_{\Omega} \kappa_j [u_R \cL_j (x,t) u_R+ u_I \cL_j (x,t) u_I] \D x \D t
	= \lambda_R \int_{0}^{T}\int_{\Omega} \left(\vert u_R\vert^2 + \vert  u_I\vert^2\right) \D x \D t.
	\end{equation}	
	
	Let us postpone the proof of these claims, and continue the proof of (ii). 
	It then follows from Claim 1 that
	\begin{equation}
	\begin{aligned}
	&-\int_{0}^{T}\int_{\Omega} \kappa_j  [u_R \cL_j (x,t) u_R+ u_I \cL_j (x,t) u_I] \D x \D t\\
	=&
	\int_{0}^{T}\int_{\Omega} \kappa_j \sum_{p,q=1}^{N} a_{pq}^{j} 
	\left[(u_R)_{x_p} (u_R)_{x_q} + (u_I)_{x_p} (u_I)_{x_q}\right] \D x \D t\\
	\geq& \kappa_j \underline{a} \int_{0}^{T}\int_{\Omega} 
	\left(\vert\nabla u_R\vert^2 + \vert \nabla u_I\vert^2\right) \D x \D t.
	\end{aligned}
	\end{equation}
	By Poincar\'e's inequality (see, e.g., \cite[Theorem 5.8.1]{evans2010partial}) and Claim 2, we have
	$$
	\int_{\Omega} \vert u_R \vert^2 \D x \leq C_1 	\int_{\Omega} \vert \nabla u_R\vert^2 \D x,~
	\int_{\Omega} \vert u_I \vert^2 \D x \leq C_1 	\int_{\Omega} \vert \nabla u_I\vert^2 \D x,~\forall t \in \Real,
	$$
	for some $C_1>0$ only dependent on $\Omega$. It then follows that
	$$
	\begin{aligned}
	\kappa_j \underline{a} \int_{0}^{T}\int_{\Omega} 
	\left(\vert u_R\vert^2 + \vert  u_I\vert^2\right) \D x \D t 
	&\leq
	C_1\kappa_j \underline{a} \int_{0}^{T}\int_{\Omega} 
	\left(\vert\nabla u_R\vert^2 + \vert \nabla u_I\vert^2\right) \D x \D t\\
	&\leq -C_1\int_{0}^{T}\int_{\Omega} \kappa_j [u_R \cL_j (x,t) u_R+ u_I \cL_j (x,t) u_I] \D x \D t\\
	&= C_1\lambda_R \int_{0}^{T}\int_{\Omega} \left(\vert u_R\vert^2 + \vert  u_I\vert^2\right) \D x \D t.
	\end{aligned}
	$$
	Here the last equality follows from Claim 3.
	Thus, $\kappa_j \underline{a}  \leq  C_1\lambda_R$.  We choose $\gamma_0\in(0,\underline{a}C_1^{-1})$ independent of $j$ such that $\gamma_0 \kappa_j  \leq \lambda_R$, which yields  the desired conclusion.
	
	Now let us return to the proof of three claims above. Clearly,  Claim 1  
	follows from the  boundary condition of \eqref{equ:eig:i:kappa}.
	To prove Claim 2, we integrate \eqref{equ:eig:i:kappa} over $\Omega$ to obtain 
	\begin{equation}\label{equ:tilde:RI}
		\begin{cases}
		\frac{\D }{\D t} \tilde{u}_R= \lambda_{R} \tilde{u}_{R} -  \lambda_{I} \tilde{u}_{I},\\
		\frac{\D }{\D t} \tilde{u}_I= \lambda_{I} \tilde{u}_{R} +\lambda_{R} \tilde{u}_{I}.
		\end{cases}
	\end{equation}
	Multiplying the first equation of \eqref{equ:tilde:RI} by $\tilde{u}_R$ and the second one by $\tilde{u}_I$ and then adding them together, we have
	$$
	\frac{\D }{\D t} \left(\tilde{u}_R^2 +\tilde{u}_I^2\right)=  2\lambda_{R} \left(\tilde{u}_{R}^2 +\tilde{u}_{I}^2\right).
	$$
	Integrating the above equation over $[0,T]$ gives rise to
	$$
	\int_{0}^{T} \left(\tilde{u}_R^2 +\tilde{u}_I^2\right)\D t=
	\frac{1}{2\lambda_{R}}
	\left[\tilde{u}_R^2 (T)+\tilde{u}_I^2(T)-
	\tilde{u}_R^2 (0)-\tilde{u}_I^2(0)\right]=0.
	$$
	This proves Claim 2.  In order to verify Claim 3,  we	
	 multiply  \eqref{equ:eig:i:kappa}  by $\bar{u}=u_{R}- \bmi u_{I}$ and integrate it  over $[0,T] \times \Omega$ to obtain
	\begin{equation}\label{equ:u:int}
		\int_{0}^{T}\int_{\Omega}\frac{\partial u}{\partial t} \bar{u} \D x \D t= \int_{0}^{T}\int_{\Omega} \bar{u}\kappa_j \cL_j (x,t) u \D x \D t + \lambda \int_{0}^{T}\int_{\Omega} u \bar{u} \D x \D t.
	\end{equation}
	An easy computation yields 
	\begin{equation}\label{equ:u:image}
	\int_{0}^{T}\int_{\Omega}\frac{\partial u}{\partial t} \bar{u} \D x \D t
	=\bmi \int_{0}^{T}\int_{\Omega} \frac{\partial u_I}{\partial t} u_R-   \frac{\partial u_R}{\partial t}u_I \D x \D t,
	\end{equation}
	and 
	\begin{equation}\label{equ:u:baru}
	\int_{0}^{T}\int_{\Omega} u \bar{u} \D x \D t
	= \int_{0}^{T}\int_{\Omega} 
	\left(\vert u_R\vert^2 + \vert  u_I\vert^2\right) \D x \D t.
	\end{equation}
	On the other hand, we see  from Claim 1  that
	$$
	\int_{0}^{T}\int_{\Omega} u_I \cL_j (x,t) u_R \D x \D t= \int_{0}^{T}\int_{\Omega} u_R \cL_j (x,t) u_I \D x \D t,
	$$
	which implies that
	\begin{equation}\label{equ:u:real}
		\int_{0}^{T}\int_{\Omega} \bar{u}\kappa_j \cL_j (x,t) u \D x \D t= \int_{0}^{T}\int_{\Omega} \kappa_j  [u_R \cL_j (x,t) u_R+ u_I \cL_j (x,t) u_I] \D x \D t.
	\end{equation}
		Thanks to \eqref{equ:u:image}, \eqref{equ:u:baru} and \eqref{equ:u:real}, \eqref{equ:int:lambda_R} follows from the real part of \eqref{equ:u:int}.
\end{proof}

Now we are ready to prove the main result of this section.
\begin{theorem}\label{thm:nonlinear:convergence:infty}
	Assume that {\rm (H1)--(H5)} hold. Then $\bmw_{\bmkp} (x,t) \rightarrow \tilde{\bmw}_{\infty}(t)$ uniformly on $\oOmega \times \Real$ as $\min_{1 \leq i \leq n}  \kappa_i  \rightarrow +\infty$.
\end{theorem}
\begin{proof}
	For convenience, let $\underline{\kappa}:=\min_{1 \leq i \leq n} \kappa_i$.
	With Lemma \ref{lem:nonlinear:convergence:infty:tilde} and the triangle inequality,  it suffices to show 
	$\hat{\bmw}_{\bmkp}(x,t) \rightarrow 0$ uniformly on $\oOmega \times \Real$ as $\underline{\kappa} \rightarrow +\infty$.
	Our arguments are motivated by  \cite{hale1986large,hale1987varying,hale1989shadow}. Choose $p>n$ and $\beta \in (\frac{1}{2}+\frac{n}{2 p},1)$.
	For each $i =1,2, \cdots,n$, let $Z_i=L^{p} (\Omega)$, $D(\cL_{i}(\cdot,0)):=W^{2,p}(\Omega)$, and let $Z_{\beta,i}$ be the fractional power space defined by $\cL_{i}(\cdot,0)$ (see, e.g., \cite{hess1991periodic,daners1992abstract,pazy1983semigroups}). Let
	$$
	Z^{n}:= Z_1 \times Z_2 \times \cdots \times Z_n 
	\text{ and }
	Z_{\beta}^{n}:= Z_{\beta,1} \times Z_{\beta,2} \times \cdots \times Z_{\beta,n}
	$$
	be two Banach spaces with the maximum norm 
	$$
	\Vert \bmu \Vert_{Z^{n}} = \max_{1 \leq i \leq n} \Vert u_i \Vert_{Z_{i}}, \text{ and }
	\Vert \bmu \Vert_{Z_{\beta}^{n}} = \max_{1 \leq i \leq n} \Vert u_i \Vert_{Z_{\beta,i}},
	$$
	where $\bmu=(u_1,u_2,\cdots,u_n)$. For each $i =1,2, \cdots,n$, let 
	$$
	P_i:= \left\{ u \in Z_i: \int_{\Omega} u(x) \D x=0 \right\} \text{ and }
	P_{\beta,i}:= \left\{ u \in Z_{\beta,i}: \int_{\Omega} u(x) \D x=0 \right\}.
	$$
	It then follows that 
	\begin{equation}\label{equ:Zi}
		Z_i=\Real \oplus P_i, \text{ and } Z_{\beta,i}=\Real \oplus P_{\beta,i},
		\, \, \forall  1\leq i \leq n.
	\end{equation}
	For convenience, we write 
	$\Vert u \Vert_{\infty} :=\max_{x \in \oOmega} \vert u(x) \vert$ for any $u \in C(\oOmega)$. According to \cite[Corollary 4.17]{daners1992abstract}, there exists $\hat{c}>0$ independent of $i$ such that $\Vert u \Vert_{\infty} \leq \hat{c} \Vert u \Vert_{Z_{\beta,i}}$ for any $u \in Z_{\beta,i}$.
	Taking the average of the $i$-th equation of  \eqref{equ:sys:main:nonlinear} over $\Omega$, we have
	\begin{equation}\label{equ:w:tilde:diff:2}
	\frac{\D}{\D t}\tilde{w}_{\bmkp,i}(t)=
	\frac{1}{\vert \Omega \vert }\int_{\Omega} G_i(x,t,\bmw_{\bmkp}^T(x,t)) \D x,~ \forall t \in \Real.
	\end{equation}
	We subtract \eqref{equ:w:tilde:diff:2} from the $i$-th equation of \eqref{equ:sys:main:nonlinear} to obtain
	\begin{equation}\label{equ:w:hat}
	\begin{aligned}
	\frac{\partial}{\partial t}\hat{w}_{\bmkp,i}(x,t)
	=& \kappa_i  \cL_i (x,t)\hat{w}_i 
	+ G_i(x,t,\bmw_{\bmkp}^T(x,t))
	-\frac{1}{\vert \Omega \vert }\int_{\Omega} G_i(y,t,\bmw_{\bmkp}^T(y,t)) \D y\\
	=& \kappa_i  \cL_i (x,t)\hat{w}_i + g_{i,1}(x,t,\bmw_{\bmkp},\tilde{\bmw}_{\bmkp})+ g_{i,2}(x,t,\bmw_{\bmkp},\tilde{\bmw}_{\bmkp})+ g_{i,3}(x,t,\bmw_{\bmkp},\tilde{\bmw}_{\bmkp}),
	\end{aligned}
	\end{equation}
	subject to the Neumann boundary condition,
	where 
	$$
	\begin{aligned}
	&g_{i,1}(x,t,\bmw_{\bmkp},\tilde{\bmw}_{\bmkp}):=G_i(x,t,\bmw_{\bmkp}^T(x,t)) - G_i(x,t,\tilde{\bmw}_{\bmkp}^T(t)),\\
	&g_{i,2}(x,t,\bmw_{\bmkp},\tilde{\bmw}_{\bmkp}):=G_i(x,t,\tilde{\bmw}_{\bmkp}^T(t))  - \vert \Omega \vert^{-1} \int_{\Omega} 
	G_i(y,t,\tilde{\bmw}_{\bmkp}^T(t)) \D y,\\
	&g_{i,3}(x,t,\bmw_{\bmkp},\tilde{\bmw}_{\bmkp}):=  
	\vert \Omega \vert^{-1} \int_{\Omega}  G_i(y,t,\tilde{\bmw}_{\bmkp}^T(t))    \D y-\vert \Omega \vert^{-1} \int_{\Omega} 
	G_i(y,t,\bmw_{\bmkp}^T(y,t))  \D y.
	\end{aligned}
	$$
	
	Since $\bmw_{\bmkp}^T(x,t)\in \cQ$ and $\hat{\bmw}_{\bmkp}^T(x,t) \in \cQ$, $\forall (x,t)\in \oOmega \times \Real$ and $G_i \in C^1(\oOmega\times\Real\times \Real^n) $, there exists $\hat{L}$ independent of $x$, $t$, $\bmkp$ and $i$ such that for any $ (x,t)\in \oOmega \times \Real$,
	$$
	\begin{aligned}
	&	\vert g_{i,1}(x,t,\bmw_{\bmkp},\tilde{\bmw}_{\bmkp}) \vert 
	\leq \hat{L} \max_{1 \leq i \leq n} \vert w_{\bmkp,i}(x,t) - \tilde{w}_{\bmkp,i}(t) 
	\vert\\
	\leq &
	\max_{1 \leq i \leq n}\hat{L} \Vert \hat{w}_{\bmkp,i}(\cdot,t) \Vert_{\infty} 
	\leq \hat{L} \hat{c} \max_{1 \leq i \leq n}  \Vert \hat{w}_{\bmkp,i}(\cdot,t) \Vert_{Z_{\beta,i}}
	= \hat{L} \hat{c} \Vert \hat{\bmw}_{\bmkp} (\cdot,t)\Vert_{Z_{\beta}^{n}},
	\end{aligned}
	$$
	and hence,
	$$
	\vert g_{i,3}(x,t,\bmw_{\bmkp},\tilde{\bmw}_{\bmkp}) \vert
	\leq \hat{L} \hat{c}   \Vert \hat{\bmw}_{\bmkp}(\cdot,t) \Vert_{Z_{\beta}^{n}}.
	$$
	Moreover, there exists $C_1>0$ such that
	$$
	\vert	g_{i,2}(x,t,\bmw_{\bmkp},\tilde{\bmw}_{\bmkp})\vert \leq C_1.
	$$
	For each $i=1,2,\cdots,n$, let $\{U_{\kappa_i,i} (t,s) : t \geq s \}$ be the evolution family on $Z_i$ of 
	$$
	\begin{cases}
	\frac{\partial u}{\partial t}= \kappa_i \cL_i (x,t)u, &x \in \Omega,~ t>0,\\
	\cB_i u=0, & x \in \partial \Omega, ~ t>0.
	\end{cases}
	$$
	Here, $\cB_i$ is the Neumann boundary condition given in \eqref{equ:Neumann}.
	Choose $\gamma_0>0$ as in Lemma \ref{lem:eig:L_i}.
	For any $\mu \in \sigma(U_{\kappa_i,i}(T,0)) \setminus \{ 1 \}$, $i=1,2,\cdots,n$, it is easy to see that $\mu =e^{-\lambda_{i,m_{\mu}} \kappa_iT}$ for some $m_{\mu} \geq 2$, which implies that $\vert \mu \vert = \vert e^{-\Realparts\lambda_{i,m_{\mu}} \kappa_iT}\vert < e^{- \gamma_0 \kappa_i T}$.
	By  \cite[Theorem 7.3]{daners1992abstract}, the decomposition forms \eqref{equ:Zi},  and  the arguments similar to those for \cite[Proposition 6.8]{daners1992abstract},  it then follows that 
	$$
	\begin{aligned}
	\Vert U_{\kappa_i,i} (t,s) u \Vert_{Z_{\beta,i}} 
	\leq C_2 e^{-\gamma_0 \kappa_i (t-s)} \Vert u \Vert_{Z_{\beta,i}},~ \forall t\geq s,~ u \in P_{\beta,i},\\
	\Vert U_{\kappa_i,i} (t,s) u \Vert_{Z_{\beta,i}} 
	\leq C_2 e^{-\gamma_0 \kappa_i (t-s)}(t-s)^{-\beta} \Vert u \Vert_{Z_{i}}, ~\forall t\geq s,~u \in P_{i},
	\end{aligned}
	$$
	for some $C_2>0$ independent of $t$ and $s$.
	On the other hand, by the constant-variation formula, for any $t >0$ and $i=1,2,\cdots,n$, we have
	$$
	\hat{w}_{\bmkp,i}(\cdot,t) = U_{\kappa_i,i}(t,0) 	\hat{w}_{\bmkp,i}(\cdot,0) +
	\int_{0}^{t} U_{\kappa_i,i}(t,s) [g_{i,2}+g_{i,1} +g_{i,3}](\cdot,s,\bmw_{\bmkp}(\cdot,s),\tilde{\bmw}_{\bmkp}(s)) \D s.
	$$
	It is easy to verify that
	$$
	\int_{\Omega} \hat{w}_{\bmkp,i} \D x=0,~
	\int_{\Omega} (g_{i,1} +g_{i,3}) \D x=0,
	\text{ and } \int_{\Omega} g_{i,2} \D x=0,~ \forall t \in \Real,
	$$
	which implies that $\hat{w}_{\bmkp,i} \in P_{\beta,i}$, $g_{i,1} +g_{i,3} \in P_i$ and $g_{i,2} \in P_i$. Thus, we obtain
	$$
	\begin{aligned}
	\Vert \hat{w}_{\bmkp,i}(\cdot,t) \Vert_{Z_{\beta,i}} 
	\leq  &C_2 e^{-\gamma_0 \kappa_i t}  \Vert \hat{w}_{\bmkp,i}(\cdot,0) \Vert_{Z_{\beta,i}}  
	+C_2 C_1 \int_{0}^{t}  e^{-\gamma_0 \kappa_i (t-s)}(t-s)^{-\beta} \D s\\
	&+ 2C_2 \hat{L} \hat{c} \int_{0}^{t}  e^{-\gamma_0 \kappa_i (t-s)}(t-s)^{-\beta}  \Vert \hat{\bmw}_{\bmkp} (\cdot,s)\Vert_{Z_{\beta}^{n}} \D s.
	\end{aligned}
	$$
	It then follows from $\underline{\kappa}=\min_{1 \leq i \leq n} \kappa_i$ and $\gamma_0 > 0$ that
	$$
	\begin{aligned}
	\Vert \hat{\bmw}_{\bmkp}(\cdot,t) \Vert_{Z_{\beta}^{n}} 
	\leq  &C_2 e^{-\gamma_0 \underline{\kappa} t}  \Vert \hat{\bmw}_{\bmkp}(\cdot,0) \Vert_{Z_{\beta}^{n}}  
	+C_2 C_1 \int_{0}^{t}  e^{-\gamma_0 \underline{\kappa} (t-s)}(t-s)^{-\beta} \D s\\
	&+2C_2 \hat{L} \hat{c} \int_{0}^{t}  e^{-\gamma_0 \underline{\kappa} (t-s)}(t-s)^{-\beta}  \Vert \hat{\bmw}_{\bmkp}(\cdot,s) \Vert_{Z_{\beta}^{n}} \D s.
	\end{aligned}
	$$
	Choose $\gamma_1 \in (0, \gamma_0)$, and define $\zeta(t):= e^{\gamma_1 \underline{\kappa} t} \Vert \hat{\bmw}_{\bmkp}(\cdot,t) \Vert_{Z_{\beta}^{n}} $,  $\overline{\zeta}(t):=\sup\{\zeta(s): 0 \leq s \leq t \}$ and 
	$$K:= \int_{0}^{\infty} s^{-\beta}e^{-s(1- \gamma_1(\gamma_0)^{-1})} \D s.$$ 
	Then we have
	$$
	\zeta(t) \leq C_2 e^{-(\gamma_0- \gamma_1)  \underline{\kappa}  t} \zeta(0) + C_2 C_1 K(\underline{\kappa} \gamma_0)^{\beta-1} e^{ \gamma_1  \underline{\kappa}  t}  + 2 C_2 \hat{L} \hat{c} K (\underline{\kappa} \gamma_0)^{\beta-1}  \overline{\zeta}(t),~ t\geq 0, 
	$$
	and hence,
	$$
	\overline{\zeta}(t) \leq C_2 \zeta(0) + C_2 C_1 K(\underline{\kappa} \gamma_0)^{\beta-1} e^{ \gamma_1  \underline{\kappa}  t} +  2C_2 \hat{L} \hat{c} K (\underline{\kappa} \gamma_0)^{\beta-1}  \overline{\zeta}(t),~ t\geq 0.
	$$
	For any $\bmkp$ with $\kappa_i>0$, let $\xi(\bmkp):=2C_2 \hat{L} \hat{c} K (\underline{\kappa} \gamma_0)^{\beta-1}$.
	Notice that $\xi(\bmkp) \rightarrow 0$ as $\underline{\kappa} \rightarrow +\infty$. From now on, we assume that $\underline{\kappa}$ is large enough such that $\xi(\bmkp)<\frac{1}{2}$, and hence, $(1-\xi(\bmkp))^{-1} \leq 2$. This leads to
	$$
	\begin{aligned}
	\overline{\zeta}(t) 
	&\leq(1-\xi(\bmkp))^{-1} [C_2  \zeta(0) + C_2 C_1 K(\underline{\kappa} \gamma_0)^{\beta-1} e^{ \gamma_1  \underline{\kappa}  t} ]\\
	&\leq 2 [C_2 \zeta(0) + C_2 C_1 K(\underline{\kappa} \gamma_0)^{\beta-1} e^{ \gamma_1  \underline{\kappa}  t}].
	\end{aligned}
	$$
	Thus, we conclude that 
	$$
	 \Vert \hat{\bmw}_{\bmkp}(\cdot,t) \Vert_{Z_{\beta}^{n}} \leq  e^{ -\gamma_1  \underline{\kappa}  t} 	\overline{\zeta}(t) \leq 2 [C_2 \zeta(0) e^{ -\gamma_1  \underline{\kappa}  t}  + C_2 C_1 K(\underline{\kappa} \gamma_0)^{\beta-1}],
	$$
	and hence,
	$$
	\limsup_{t \rightarrow +\infty} \Vert \hat{\bmw}_{\bmkp}(\cdot,t) \Vert_{Z_{\beta}^{n}}  \leq2  C_2 C_1 K(\underline{\kappa} \gamma_0)^{\beta-1}.
	$$
	Since $\hat{\bmw}(x,t)$ is periodic in $t \in \Real$, it follows that
	$$
	\sup_{t \in \Real} \Vert \hat{\bmw}_{\bmkp}(\cdot,t) \Vert_{Z_{\beta}^{n}}  \leq 2   C_2 C_1 K(\underline{\kappa} \gamma_0)^{\beta-1}.
	$$
	Therefore, $\sup_{t \in \Real} \Vert \hat{\bmw}_{\bmkp}(\cdot,t) \Vert_{Z_{\beta}^{n}} \rightarrow 0 $ as   $\underline{\kappa} \rightarrow +\infty$, which implies that $\hat{\bmw}_{\bmkp}(x,t) \rightarrow 0$ uniformly on $\oOmega \times \Real$ as $\underline{\kappa} \rightarrow +\infty$. 
\end{proof}

\section{An application}\label{example}
In this section, we apply our developed theory to study the asymptotic behavior of the basic reproduction ratio for a reaction-diffusion 
model of Zika virus transmission with small and large diffusion coefficients. 

According to  \cite{fitzgibbon2017outbreak, li2020preprint}, the model is
governed by the following time-periodic reaction-diffusion system:
\begin{equation}\label{equ:sys:Hi_Vu_Vi}
	\begin{cases}
	\frac{\partial H_i}{\partial t}
	= \kappa_1  \nabla \cdot( \delta_1(x,t) \nabla H_i )   - \gamma(x,t) H_i + \sigma_1 (x,t) H_u(x) V_i, & x \in \Omega,~t>0,\\
	\frac{\partial V_u}{\partial t}
	= \kappa_2  \nabla \cdot( \delta_2(x.t) \nabla V_u ) - \sigma_2(x,t) V_u H_i   \\
	\hspace{1.1cm}+ \beta(x,t) (V_u +V_i)- \mu_1(x,t) V_u- \mu_2(x,t)(V_u +V_i) V_u, & x \in \Omega,~t>0,\\
	\frac{\partial V_i}{\partial t}
	= \kappa_2 \nabla \cdot( \delta_2(x,t) \nabla V_i ) + \sigma_2(x,t) V_u H_i \\
	\hspace{1.1cm}- \mu_1(x,t) V_i- \mu_2(x,t)(V_u +V_i) V_i,& x \in \Omega,~t>0,\\
	\frac{\partial H_i}{\partial \bm{\nu}}
	=\frac{\partial V_u}{\partial \bm{\nu}}
	=\frac{\partial V_i}{\partial \bm{\nu}}=0,
	& x \in \partial \Omega,~ t >0.
	\end{cases}
\end{equation}
Here  $\Omega$ is a bounded spatial domain with smooth boundary $\partial \Omega$,  $H_i(x,t)$, $V_i(x,t)$ and $V_u(x,t)$ are  the densities of infected hosts, susceptible vectors, and infected vectors, respectively. 
The parameters of the model are shown as in Table \ref{tab:para}.
We assume that $\delta_1$, $\delta_2$ $\in C^{1+\alpha}(\oOmega  \times \Real)$ and $H_u \in C^{\alpha}(\oOmega)$ are strictly positive for some  $0 < \alpha < 1$; the functions $\beta$, $\gamma$, $\mu_1$, $\mu_2$, $\sigma_1$ and $\sigma_2$  are $T$-periodic, strictly positive and H\"{o}lder continuous on $\oOmega \times \Real$. 
\begin{table}[ht]
	\caption{Biological interpretations of parameters}
	\centering\label{tab:para}
	\begin{tabular}{cl}
		\hline 
		$\kappa_1$ and $\delta_1(x,t)$ &  Diffusion coefficient for hosts.\\
		$\kappa_2$ and $\delta_2(x,t)$ & Diffusion coefficient for vectors.\\
		$H_u(x)$ & Densities of susceptible hosts at location $x$.\\
		$\gamma(x,t)$ & Loss rate of infected hosts at location $x$ and time $t$. \\
		$\beta(x, t)$ & Breeding rate of vectors at location $x$ and time $t$.\\
		$\sigma_1(x,t)$ &  Transmission rate for susceptible hosts at location $x$ and time $t$.\\
		$\sigma_2(x,t)$& Transmission rate for susceptible vectors at location $x$ and time $t$.\\
		$\mu_1(x,t)$ & Natural mortality rate of vectors at location $x$ and time $t$.\\
		$\mu_2(x,t)$ & Density dependent loss rate of vectors at location $x$ and time $t$.\\
		\hline 
	\end{tabular} 
\end{table}

Let $V=V_i + V_u$ be the total densities of the vector population. It then follows that
\begin{equation}\label{equ:sys:V}
	\frac{\partial V}{\partial t}
	= \kappa_2 \nabla \cdot( \delta_2(x,t) \nabla V ) + \beta(x,t) V- \mu_1(x,t) V- \mu_2(x,t) V^2,~ x \in \Omega,~t>0,
\end{equation}
subject to the Neumann boundary condition.  We further assume that $$\beta (x,t) - \mu_1 (x,t) >0, ~\forall (x,t) \in \oOmega \times \Real.$$
Since \eqref{equ:sys:V}  is a standard periodic parabolic logistic-type equation, by \cite[Theorem 28.1]{hess1991periodic} or \cite[Theorem 3.1.5]{zhao2017dynamical}, system \eqref{equ:sys:V} admits a globally stable positive $T$-periodic solution $V^*_{\kappa_2}(x,t)$, and hence,
any positive solution $V(x,t)$ of it satisfies
$$
\lim\limits_{t \rightarrow +\infty} \Vert V(\cdot,t) - V^*_{\kappa_2}(\cdot,t) \Vert_{C(\oOmega)} =0.
$$

For  convenience, we  transform system \eqref{equ:sys:Hi_Vu_Vi} into
an equivalent one:
\begin{equation}\label{equ:sys:Hi_Vi_V}
\begin{cases}
\frac{\partial H_i}{\partial t}
= \kappa_1  \nabla \cdot( \delta_1(x,t) \nabla H_i )   - \gamma(x,t) H_i + \sigma_1 (x,t) H_u(x) V_i, & x \in \Omega,~t>0,\\
\frac{\partial V_i}{\partial t}
= \kappa_2  \nabla \cdot( \delta_2(x,t) \nabla V_i ) + \sigma_2(x,t) (V- V_i) H_i \\
\hspace{1.1cm}- \mu_1(x,t) V_i- \mu_2(x,t)V V_i,& x \in \Omega,~t>0,\\
\frac{\partial V}{\partial t}
= \kappa_2 \nabla \cdot( \delta_2(x,t) \nabla V ) + \beta(x,t) V- \mu_1(x,t) V- \mu_2(x,t) V^2,& x \in \Omega,~t>0,\\
\frac{\partial H_i}{\partial \bm{\nu}}
=\frac{\partial V_i}{\partial \bm{\nu}}
=\frac{\partial V}{\partial \bm{\nu}}=0,
& x \in \partial \Omega,~ t >0.
\end{cases}
\end{equation}
Linearizing the system \eqref{equ:sys:Hi_Vi_V} at its disease-free periodic solution  $(0,0,V^*_{\kappa_2}(x,t))$, we obtain the following periodic reaction-diffusion system of two infected variables:
\begin{equation}
	\begin{cases}
	\frac{\partial H_i}{\partial t}
	= \kappa_1 \nabla \cdot( \delta_1(x,t) \nabla H_i )   - \gamma(x,t) H_i + \sigma_1 (x,t) H_u(x) V_i, & x \in \Omega,~t>0,\\
	\frac{\partial V_i}{\partial t}
	= \kappa_2 \nabla \cdot( \delta_2(x,t) \nabla V_i ) + \sigma_2(x,t) V^{*}_{\kappa_2}(x,t) H_i \\
	\hspace{1.1cm}- \mu_1(x,t) V_i- \mu_2(x,t)V^{*}_{\kappa_2}(x,t) V_i,& x \in \Omega,~t>0,\\
	\frac{\partial H_i}{\partial \bm{\nu}}
	=\frac{\partial V_i}{\partial \bm{\nu}}
	=0,
	& x \in \partial \Omega,~ t >0.
	\end{cases}
\end{equation}

Now we choose $n=2$. It then follows that $X:=C(\oOmega,\Real^2)$, 
$$
\bX:=\{ \bmu \in C( \Real,X): \bmu(t)=\bmu(t+T), ~t \in \Real\},
$$
and
$$
\bP:=\{ \bmu \in C(\Real,\Real^2): \bmu(t)=\bmu(t+T), ~t \in \Real \}.
$$
Let $\{ \Phi_{\kappa_1,\kappa_2}(t,s): t \geq s \}$ be the evolution family on $X$ of 
\begin{equation}
\begin{cases}
\frac{\partial H_i}{\partial t}
= \kappa_1 \nabla \cdot( \delta_1(x,t) \nabla H_i )   - \gamma(x,t) H_i , & x \in \Omega,~t>0,\\
\frac{\partial V_i}{\partial t}
= \kappa_2 \nabla \cdot( \delta_2(x,t) \nabla V_i ) - \mu_1(x,t) V_i- \mu_2(x,t)V^{*}_{\kappa_2}(x,t) V_i,& x \in \Omega,~t>0,\\
\frac{\partial H_i}{\partial \bm{\nu}}
=\frac{\partial V_i}{\partial \bm{\nu}}
=0,
& x \in \partial \Omega,~ t >0.
\end{cases}
\end{equation}
Define two positive bounded linear operators $F_{\kappa_2}(t):X \rightarrow X$ and $\bQ_{\kappa_1,\kappa_2}:\bX \rightarrow \bX$ by 
$$
\left[
F_{\kappa_2}(t) 
\left( \begin{array}{cc}
\phi_1\\\phi_2
\end{array}\right)
\right](x)
:=\left( \begin{array}{cc}
\sigma_1 (x,t) H_u(x) \phi_2(x)\\
\sigma_2(x,t) V^{*}_{\kappa_2}(x,t) \phi_1(x)
\end{array}\right),~x \in \oOmega,~ t \in \Real,~ \phi=\left( \begin{array}{cc}
\phi_1\\\phi_2
\end{array}\right)\in X,
$$
and
$$
[\bQ_{\kappa_1,\kappa_2} \bmu](t): = \int_{0}^{+\infty} \Phi_{\kappa_1,\kappa_2}(t,t-s)F_{\kappa_2}(\cdot,t-s)\bmu(t-s) \D s, ~ t \in \Real,~ \bmu \in \bX.
$$
It then follows that  the basic reproduction ratio $\R (\kappa_1,\kappa_2):=r(\bQ_{\kappa_1,\kappa_2})$.  

Next we study the asymptotic behavior of $\R (\kappa_1,\kappa_2)$ as $\max(\kappa_1,\kappa_2) \rightarrow 0$ and 
$\min(\kappa_1,\kappa_2) \rightarrow +\infty$, respectively. 
By \cite[Theorem 3.1.2]{zhao2017dynamical}, it follows that for each $x \in \oOmega$, the following scalar periodic ODE equation
$$
\frac{\partial V}{\partial t}
= \beta(x,t) V- \mu_1(x,t) V- \mu_2(x,t) V^2,~t>0
$$
admits a globally stable positive $T$-periodic solution $V_0(x,t)$. Moreover, $V_0(x,t)$ is continuous on $\oOmega \times \Real$. Indeed, by arguments similar to those in the proof of Lemma \ref{lem:nonlinear:convergence:infty:tilde}, for any $x_0 \in \oOmega$, $V_0(x,t)$ converges uniformly to $V_0(x_0,t)$ on $[0,T]$ as $x$ goes to $x_0$.
Define
$$
\tilde{\beta}(t):=\vert \Omega \vert^{-1} \int_{\Omega} \beta (x,t) \D x,~
\tilde{\gamma}(t):=\vert \Omega \vert^{-1} \int_{\Omega} \gamma (x,t) \D x,~
\tilde{\sigma}_2(t):=\vert \Omega \vert^{-1} \int_{\Omega} \sigma_2 (x,t) \D x.
$$
$$
\tilde{\mu}_1(t):=\vert \Omega \vert^{-1} \int_{\Omega} \mu_1 (x,t) \D x, \text{ and }
\tilde{\mu}_2(t):=\vert \Omega \vert^{-1} \int_{\Omega} \mu_2 (x,t) \D x.
$$ 
Using  \cite[Theorem 3.1.2]{zhao2017dynamical} again, we see that the following scalar periodic ODE system
$$
\frac{\partial V}{\partial t}
= \tilde{\beta}(t) V- \tilde{\mu}_1(t) V- \tilde{\mu}_2(t) V^2,~t>0
$$
admits a globally stable positive $T$-periodic solution $\tilde{V}_{\infty}(t)$. 

It is easy to verify assumptions (H1)--(H5) hold true. Thus, Theorems \ref{thm:nonlinear:convergence:0} and \ref{thm:nonlinear:convergence:infty}
imply 
$$
\lim\limits_{\kappa_2 \rightarrow 0} \Vert V^*_{\kappa_2}(\cdot,t) - V_{0}(\cdot,t)  \Vert_{C(\oOmega)} =0,~
\lim\limits_{\kappa_2 \rightarrow +\infty} \Vert V^*_{\kappa_2}(\cdot,t) - \tilde{V}_{\infty}(t)  \Vert_{C(\oOmega)} =0.
$$
For each $x \in \oOmega$, let $\{\Gamma_{x,0}(t,s): t\geq s\}$ be the evolution family on $\Real^2$ of 
\begin{equation}
\begin{cases}
\frac{\partial H_i}{\partial t}
=   - \gamma(x,t) H_i , & x \in \oOmega,~t>0,\\
\frac{\partial V_i}{\partial t}
=  - \mu_1(x,t) V_i- \mu_2(x,t)V_{0}(x,t) V_i,& x \in \oOmega,~t>0,\\
\end{cases}
\end{equation}
and define
$$
\cF_{0}(x,t)\left( \begin{array}{cc}
\phi_1\\\phi_2
\end{array}\right)
:=
\left( 
\begin{array}{cc}
\sigma_1 (x,t) H_u(x) \phi_2\\
\sigma_2(x,t) V_{0}(x,t) \phi_1
\end{array}\right),~(x,t) \in \oOmega \times \Real,~ \phi=
\left( \begin{array}{cc}
\phi_1\\\phi_2
\end{array}\right) \in \Real^2.
$$
Let $\{\tilde{\Gamma}_{\infty}(t,s): t\geq s\}$ be the evolution family on $\Real^2$ of 
\begin{equation}
\begin{cases}
\frac{\partial H_i}{\partial t}
=   - \tilde{\gamma}(t) H_i , & x \in \oOmega,~t>0,\\
\frac{\partial V_i}{\partial t}
=  - [\tilde{\mu}_1 (t)+ \tilde{\mu}_2 (t) \tilde{V}_{\infty}(t) ]V_i,& x \in \oOmega,~t>0,\\
\end{cases}
\end{equation}
and define
$$
\tilde{\cF}_{\infty}(t)\left( \begin{array}{cc}
\phi_1\\\phi_2
\end{array}\right):=
\left( \begin{array}{cc}
\tilde{f}_{12}(t) \phi_2\\
\tilde{\sigma}_{2}(t) \tilde{V}_{\infty}(t) \phi_1
\end{array}\right),~ t \in \Real,~ \phi=\left( \begin{array}{cc}
\phi_1\\\phi_2
\end{array}\right) \in \Real^2,
$$
where
$$
\tilde{f}_{12}(t):=\vert \Omega \vert^{-1} \int_{\Omega} \sigma_1 (x,t) H_u(x) \D x,~t \in \Real.
$$
For each $x \in \oOmega$, we introduce a bounded linear positive operator $Q_{x,0} : \bP \rightarrow \bP$ by
$$
[Q_{x,0} \bmu] (t):=\int_{0}^{+\infty} \Gamma_{x,0}(t,t-s) \cF_{0}(x,t-s) \bmu(t-s) \D s, ~t \in \Real,~ \bmu \in \bP,
$$
and define a bounded linear positive operator $\tilde{Q}_{\infty} : \bP \rightarrow \bP$ by 
$$
[\tilde{Q}_{\infty} \bmu] (t):= \int_{0}^{+\infty} \tilde{\Gamma}_{\infty}(t,t-s)\tilde{\cF}_{\infty}(t-s)\bmu(t-s) \D s, ~ t \in \Real, ~\bmu \in \bP.
$$

Let us define $R_0(x,0):= r(Q_{x,0})$, $\forall x \in \oOmega$, and $\tilde{R}_{0}(\infty):=r(\tilde{Q}_{\infty})$. 
By Theorem \ref{thm:R0:infty:Neumann} with  $\bmkp={\rm diag}(\kappa_1,\kappa_2)$, $\chi=\kappa_2$, and $\chi_0=0$, and $\bmkp={\rm diag}(\kappa_1,\kappa_2)$, $\chi=\frac{1}{\kappa_2}$, and $\chi_0=0$, respectively,  it then follows that 
$$
\lim\limits_{\max(\kappa_1,\kappa_2) \rightarrow 0} \R(\kappa_1,\kappa_2)=\max_{ x \in \oOmega} R_0(x,0), \text{ and }
\lim\limits_{\min(\kappa_1,\kappa_2) \rightarrow +\infty} \R(\kappa_1,\kappa_2)= \tilde{R}_{0}(\infty).
$$

\noindent
{\bf Acknowledgements.}
L. Zhang's research is supported in part by the National Natural Science Foundation of China(11901138) and  the Natural Science Foundation of Shandong Province(ZR2019QA006), and X.-Q. Zhao’s research is supported in part by the NSERC of Canada.


\end{document}